\documentclass[12pt]{amsart}
\usepackage{a4wide}
\usepackage{amsmath}
\usepackage[utf8]{inputenc}
\usepackage{amssymb}
\usepackage{amsopn}
\usepackage{epsfig}
\usepackage{amsfonts}
\usepackage{latexsym}
\usepackage{graphicx}
\usepackage{enumerate}
\usepackage[dvipsnames]{xcolor}
\usepackage{mathrsfs }
\usepackage{tikz}
\usetikzlibrary{arrows,automata}

\newtheorem{theorem}{Theorem}[section]
\newtheorem{lemma}[theorem]{Lemma}
\newtheorem{proposition}[theorem]{Proposition}
\newtheorem{corollary}[theorem]{Corollary}

\theoremstyle{definition}
\newtheorem{definition}[theorem]{Definition}
\newtheorem{example}[theorem]{Example}

\theoremstyle{remark}
\newtheorem{remark}[theorem]{Remark}

\numberwithin{equation}{section}

\newcommand{\N}{\ensuremath{\mathbb{N}}}
\renewcommand{\L}{\ensuremath{\mathbb{L}}}
\renewcommand{\S}{\ensuremath{\mathbb{S}}}

\newcommand{\cF}{\mathcal{F}}

\newcommand{\cL}{\mathcal{L}}

\renewcommand{\c}{ {\mathbf{c}}}
\renewcommand{\a}{ {\mathbf{a}}}
\renewcommand{\d}{ {\mathbf{d}}}
\renewcommand{\b}{ {\mathbf{b}}}

\renewcommand{\t}{ {\mathbf{t}}}
\newcommand{\s}{ {\mathbf{s}}}
\newcommand{\cs}{ {\mathbf{S}}}
\newcommand{\ca}{ {\mathbf{A}}}
\renewcommand{\r}{ {\mathbf{r}}}
\newcommand{\E}{\mathscr{E}}

\newcommand{\BB}{\mathscr{B}}

\renewcommand{\v}{\mathbf{v}}
\newcommand{\w}{\mathbf{w}}
\newcommand{\x}{\mathbf{x}}

\newcommand{\z}{\mathbf{z}}

\newcommand{\set}[1]{\left\{#1\right\}}
\newcommand{\la}{\lambda}

\newcommand{\Ga}{\Gamma}
\newcommand{\ep}{\varepsilon}
\newcommand{\f}{\infty}
\newcommand{\de}{\delta}

\newcommand{\al}{\alpha}
\newcommand{\lle}{\preccurlyeq}
\newcommand{\lge}{\succcurlyeq}
\renewcommand{\a}{ \mathbf{a}}
\newcommand{\si}{\sigma}
\newcommand{\La}{\Lambda}

\newcommand{\K}{\mathcal K}

\title[The $\beta$-transformation with a hole at $0$]{The $\beta$-transformation with a hole at $0$: the general case}

\author[P. Allaart]{Pieter Allaart}
\address[P. Allaart]{Mathematics Department, University of North Texas, 1155 Union Cir \#311430, Denton, TX 76203-5017, U.S.A.}
\email{allaart@unt.edu}

\author[D. Kong]{Derong Kong}
\address[D. Kong]{College of Mathematics and Statistics, Center of Mathematics,  Chongqing University, Chongqing 401331, People's Republic of China.}
\email{derongkong@126.com}

\begin{document}


\begin{abstract}
Given $\beta>1$, let $T_\beta$ be the $\beta$-transformation on the unit circle $[0,1)$, defined by $T_\beta(x)=\beta x-\lfloor \beta x\rfloor$. For each $t\in[0,1)$ let $K_\beta(t)$ be the survivor set consisting of all $x\in[0,1)$ whose orbit $\{T^n_\beta(x): n\ge 0\}$ never enters the interval $[0,t)$.  Kalle et al.~[{\em Ergodic Theory Dynam.~Systems} {\bf 40} (2020), no.~9, 2482--2514] considered the case $\beta\in(1,2]$. They studied the set-valued bifurcation set $\E_\beta:=\set{t\in[0,1): K_\beta(t')\ne K_\beta(t)~\forall t'>t}$ and proved that the Hausdorff dimension function $t\mapsto\dim_H K_\beta(t)$ is a non-increasing Devil's staircase. In a previous paper [{\em Ergodic Theory Dynam.~Systems} {\bf 43} (2023), no.~6, 1785--1828] we determined, for all $\beta\in(1,2]$, the critical value $\tau(\beta):=\min\{t>0: \eta_\beta(t)=0\}$. The purpose of the present article is to extend these results to all $\beta>1$. In addition to calculating $\tau(\beta)$, we show that (i) the function $\tau: \beta\mapsto\tau(\beta)$ is left continuous on $(1,\f)$ with right-hand limits everywhere, but has countably infinitely many discontinuities; (ii) $\tau$ has no downward jumps; and (iii) there exists an open set $O\subset(1,\f)$, whose complement $(1,\f)\setminus O$ has zero Hausdorff dimension, such that $\tau$ is real-analytic, strictly convex and strictly decreasing on each connected component of $O$. 
We also prove several topological properties of the bifurcation set $\E_\beta$. The key to extending the results from $\beta\in(1,2]$ to all $\beta>1$ is an appropriate generalization of the Farey words that are used to parametrize the connected components of the set $O$. Some of the original proofs from the above-mentioned papers are simplified. 
\end{abstract}


\maketitle

\section{Introduction}

For a real number $\beta>1$, consider the $\beta$-transformation
\[
T_\beta: [0,1)\to[0,1), \quad x\mapsto\beta x-\lfloor \beta x\rfloor,
\]
where $\lfloor r\rfloor$ denotes the integer part of a real number $r$.
Given a number $t\in[0,1)$, we are interested in the {\em survivor set}
\[
K_\beta(t):=\set{x\in[0,1): T_\beta^n (x)\ge t\ \forall n\ge 0};
\]
that is, the set of points whose forward orbit never enters the ``hole" $[0,t)$ \footnote{Traditionally in open dynamical systems, one takes the hole to be open. But because here $0$ is both an endpoint of the hole and an endpoint of the domain, it is more convenient to include this point with the hole; this affects the survivor set by at most countably many points.}.
Clearly, the set-valued function $t\mapsto K_\beta(t)$ is non-increasing in $t$, and we may wish to study the {\em bifurcation set}
\begin{equation} \label{eq:bifurcation-set}
\E_\beta:=\set{t\in[0,1): K_\beta(t')\ne K_\beta(t)~\forall t'>t}.
\end{equation}
Kalle et al.~\cite{Kalle-Kong-Langeveld-Li-18} considered the case $\beta\in(1,2]$, and proved several properties of the set $\E_\beta$ (denoted by $E_\beta^+$ in their paper). They furthermore showed that the dimension function $\eta_\beta: t\mapsto \dim_H K_\beta(t)$ is a decreasing devil's staircase. In particular, there is a point $t=\tau(\beta)$, called the {\em critical value}, at which $\eta_\beta(t)$ first reaches the value zero. That is,
\[
\tau(\beta):=\min\{t>0: \eta_\beta(t)=0\}.
\]
In a previous paper \cite{Allaart-Kong-2021}, the present authors determined this critical value $\tau(\beta)$ for all $\beta\in(1,2]$.

The aim of the present article is to extend the results of \cite{Kalle-Kong-Langeveld-Li-18} and \cite{Allaart-Kong-2021} to all $\beta>1$. The primary difficulty in this generalization lies in the fact that for $\beta>2$ one needs alphabets of size three or greater to understand the symbolic dynamics underlying the map $T_\beta$. The proofs in \cite{Kalle-Kong-Langeveld-Li-18} and \cite{Allaart-Kong-2021} relied heavily on the use of {\em Farey words}, which are words over the alphabet $\{0,1\}$ that satisfy a certain balancing property. It is not a priori clear how one should extend these Farey words to larger alphabets. Doing this in the right way is the key to extending the above-mentioned results.

For the critical value $\tau(\beta)$, we state here first a general, qualitative result about the graph of $\beta\mapsto\tau(\beta)$. The precise calculation of $\tau(\beta)$ splits into several cases and requires additional notation and concepts; we will present it in Section \ref{sec:critical-value}.

\begin{theorem} \label{main:critical-devils-staircase}\mbox{}
\begin{enumerate}[{\rm(i)}]
\item The function $\tau: \beta\mapsto\tau(\beta)$ is left continuous on $(1,\infty)$ with right-hand limits everywhere, but also has a countably infinite set of discontinuities;
\item $\tau$ has no downward jumps;
\item There is an open set $O\subseteq (1,\infty)$, whose complement $(1,\infty)\setminus O$ has zero Hausdorff dimension, such that $\tau$ is real-analytic, strictly convex and strictly decreasing on each connected component of $O$.
\end{enumerate}
\end{theorem}

We illustrate Theorem \ref{main:critical-devils-staircase} in Figure \ref{fig:graph-tau-more}. This result was proved in \cite{Allaart-Kong-2021} for $\beta\in(1,2]$. We will show that it continues to hold for all $\beta>1$. In the process, we also streamline some of the proofs.

\begin{figure}[h!]
 \begin{center}
\begin{tikzpicture}[xscale=4.5,yscale=10]
\draw [->] (1,0) node[anchor=north] {$1$}  -- (4.15,0) node[anchor=west] {$\beta$};
\draw [->] (1,0) node[anchor=east] {$0$} -- (1,0.8) node[anchor=south] {$\tau(\beta)$};

\draw[dashed] (2,-0.005) node[anchor=north] {$2$} -- (2,0.5)--(1,0.5)node[anchor=east] {$\frac{1}{2}$};
\draw[dashed] (3,-0.005) node[anchor=north] {$3$} -- (3,0.667)--(1,0.667)node[anchor=east] {$\frac{2}{3}$};
\draw[dashed] (4,-0.005) node[anchor=north] {$4$} -- (4,0.75)--(1,0.75)node[anchor=east] {$\frac{3}{4}$};

\draw[variable=\q,domain=1:4,dashed,red] plot({\q},{(1-1/\q)})(2.2,0.56)node[above,scale=1.2pt]{$1-\frac{1}{\beta}$};

\draw[densely dotted,variable=\q,domain=1:1.1187] plot({\q},{1-1/\q});


\draw[thick,variable=\q,domain=1.1339: 1.13871] plot({\q},({1/\q/(pow(\q,17)-1)});	

\draw[thick,variable=\q,domain=1.14003: 1.14532] plot({\q},({1/\q/(pow(\q,16)-1)});	

\draw[thick,variable=\q,domain=1.14685: 1.15271] plot({\q},({1/\q/(pow(\q,15)-1)});	

\draw[thick,variable=\q,domain=1.15449: 1.16102] plot({\q},({1/\q/(pow(\q,14)-1)});	

\draw[thick,variable=\q,domain=1.16312: 1.17045] plot({\q},({1/\q/(pow(\q,13)-1)});	

\draw[thick,variable=\q,domain=1.17295: 1.18126] plot({\q},({1/\q/(pow(\q,12)-1)});	

\draw[thick,variable=\q,domain=1.18428: 1.19379] plot({\q},({1/\q/(pow(\q,11)-1)});	

\draw[thick,variable=\q,domain=1.19749: 1.20851] plot({\q},({1/\q/(pow(\q,10)-1)});   


\draw[thick,variable=\q,domain=1.21315: 1.22611] plot({\q},({1/\q/(pow(\q,9)-1)});	

\draw[dotted,thick] (1.2285,0.1568)--(1.2285,0.186);

\draw[thick,variable=\q,domain=1.23205: 1.24758] plot({\q},({1/\q/(pow(\q,8)-1)});	

\draw[dotted,thick] (1.2507,0.167)--(1.2507,0.2005);


\draw[thick,variable=\q,domain=1.25542: 1.27444] plot({\q},({1/\q/(pow(\q,7)-1)});	

\draw[dotted,thick] (1.2787,0.1789)--(1.2787,0.2179);


\draw[thick,variable=\q,domain=1.2852: 1.3092] plot({\q},({1/\q/(pow(\q,6)-1)});	

\draw[dotted,thick] (1.3151,0.1933)--(1.3151,0.2396);

\draw[thick,variable=\q,domain=1.31663: 1.3225] plot({\q},{(1/\q+pow(\q,{-6}))/(1-pow(\q,{-11}))-1/\q}); 

\draw[thick,variable=\q,domain=1.32472: 1.35626] plot({\q},({1/\q/(pow(\q,5)-1)});	

\draw[dotted,thick] (1.365,0.211)--(1.365,0.2674);

\draw[thick,variable=\q,domain=1.35787: 1.36329] plot({\q},{pow(\q,-6)+(1/\q+pow(\q,-5))/(pow(\q,10)-1)});	

\draw[thick,variable=\q,domain=1.36759: 1.37635]plot({\q},{(1/\q+pow(\q,{-5}))/(1-pow(\q,{-9}))-1/\q});	

\draw[thick,variable=\q,domain=1.38028:1.42421] plot({\q},({1/\q/(pow(\q,4)-1)});	

\draw[dotted,thick] (1.4384,0.2336)--(1.4384,0.3048);

\draw[thick,variable=\q,domain=1.42706: 1.43532] plot({\q},{pow(\q,-5)+(1/\q+pow(\q,-4))/(pow(\q,8)-1)});	

\draw[thick,variable=\q,domain=1.44327: 1.45759] plot({\q},{(1/\q+pow(\q,{-4}))/(1-pow(\q,{-7}))-1/\q});	

\draw[thick,variable=\q,domain=1.46557:1.53259] plot({\q},({1/\q/(pow(\q,3)-1)});	

\draw[dotted,thick] (1.559,0.2639)--(1.559,0.3586);


\draw[thick,variable=\q,domain=1.5385: 1.55256 ] plot({\q}, {pow(\q,-4) +(1/\q+pow(\q,-3))/(pow(\q,6)-1)});	


\draw[thick,variable=\q,domain=1.57015: 1.5974] plot({\q},{(1/\q+pow(\q,{-3}))/(1-pow(\q,{-5}))-1/\q});	

\draw[dotted,thick] (1.6002,0.3393)--(1.6002,0.3751);

\draw[thick,variable=\q,domain=1.61803:1.73867] plot({\q},({1/\q/(pow(\q,2)-1)});	

\draw[dotted, blue]({1.61803},({1/1.61803/(pow(1.61803,2)-1)})--(1.61803,-0.01);
\draw[dotted,blue]({1.73867},({1/1.73867/(pow(1.73867,2)-1)})--(1.73867,-0.01);
\draw ({(1.61803+1.73867)/2},-0.06)node[anchor=south] {$I^{01}$};

\draw[dotted,thick] (1.8019,0.308)--(1.8019,0.4451);

\draw[thick,variable=\q,domain=1.7437: 1.75337] plot({\q}, {pow(\q,-3)+pow(\q,-5)+(1/\q+pow(\q,-2)+pow(\q,-5))/(pow(\q,6)-1)});	

\draw[thick,variable=\q,domain=1.75488: 1.78431] plot({\q}, {pow(\q,-3)+(1/\q+pow(\q,-2))/(pow(\q,4)-1)});	

\draw[dotted,thick] (1.7875,0.273)--(1.7875,0.3027);


\draw[thick,variable=\q,domain=1.78854: 1.79758] plot({\q}, {pow(\q,-3)+pow(\q,-4)+(1/\q+pow(\q,-2)+pow(\q,-4))/(pow(\q,6)-1)});	

\draw[thick,variable=\q,domain=1.8124: 1.83401] plot({\q}, {((1-pow(\q,-2-2))/(\q-1)-pow(\q,-1-2))/(1-pow(\q,-2-3))-1/\q});	

\draw[dotted,thick] (1.8351,0.4332)--(1.8351,0.4551);

\draw[thick,variable=\q,domain=1.83929: 1.9097] plot({\q}, {(pow(\q,-2)-1)/(\q-1)/(pow(\q,-3)-1)-1/\q});	

\draw[dotted,thick] (1.9213,0.4119)--(1.9213,0.4795);


\draw[thick,variable=\q,domain=1.91118: 1.91971] plot({\q}, {pow(\q,-2)+pow(\q,-4)+pow(\q,-5)+(pow(\q,-1)+pow(\q,-2)+pow(\q,-3)+pow(\q,-5))/(pow(\q,6)-1)});	

\draw[thick,variable=\q,domain=1.92213: 1.92712] plot({\q}, {((1-pow(\q,-4-2))/(\q-1)-pow(\q,-2-2))/(1-pow(\q,-4-3))-1/\q});	

\draw[thick,variable=\q,domain=1.92756: 1.96223] plot({\q}, {(pow(\q,-3)-1)/(\q-1)/(pow(\q,-4)-1)-1/\q});	

\draw[dotted,thick] (1.9647,0.4581)--(1.9647,0.491);

\draw[thick,variable=\q,domain=1.96595: 1.98274] plot({\q}, {(pow(\q,-4)-1)/(\q-1)/(pow(\q,-5)-1)-1/\q});	

\draw[thick,variable=\q,domain=1.98358: 1.99177] plot({\q}, {(pow(\q,-5)-1)/(\q-1)/(pow(\q,-6)-1)-1/\q});	



\draw[thick,variable=\q,domain=2:2.325] plot({\q},{1/(pow(\q,2)-\q)}); 
\draw[dotted,blue](2,1/2)--(2,-0.01);
\draw[dotted,blue](2.325,0.325)--(2.325,-0.01);
\draw(2.1625,-0.06)node[anchor=south] {$I^1$};

\draw[dotted,thick] (2.618,0.382)--(2.618,0.618);

\draw[thick,variable=\x,domain=0.4242:0.4275] plot({1/\x},{pow(\x,2)+(pow(\x,3)+pow(\x,4)+2*pow(\x,5))/(1-pow(\x,4))}); 

\draw[thick,variable=\q,domain=2.359:2.402] plot({\q}, {(\q+1)/pow(\q,3)+(2*pow(\q,2)+1)/(pow(\q,3)*(pow(\q,3)-1))}); 

\draw[dotted,thick] (2.4048,0.3167)--(2.4048,0.3413);

\draw[thick,variable=\q,domain=2.414:2.519] plot({\q},{1/pow(\q,2)+2/(\q*(pow(\q,2)-1))}); 
\draw[dotted, blue](2.414,0.343)--(2.414,-0.01);
\draw[dotted,blue](2.519,0.306)--(2.519,-0.01);
\draw (2.4665,-0.06)node[anchor=south] {$I^{02}$};

\draw[dotted,thick] (2.5387,0.3103)--(2.5387,0.3673);

\draw[thick,variable=\q,domain=2.547:2.589] plot({\q},{2/pow(\q,2)+(2*\q+1)/(pow(\q,2)*(pow(\q,3)-1))}); 

\draw[dotted,thick] (2.5911,0.3554)--(2.5911,0.3771);

\draw[thick,variable=\q,domain=2.592:2.608] plot({\q},{(2*\q+1)/pow(\q,3)+(2*pow(\q,2)+\q+1)/(pow(\q,3)*(pow(\q,4)-1))}); 

\draw[thick,variable=\q,domain=2.633:2.657] plot({\q},({1/\q/(pow(\q,4)-1)+1/(\q-1)}); 

\draw[thick,variable=\q,domain=2.659:2.719] plot({\q},({1/\q/(pow(\q,3)-1)+1/(\q-1)}); 

\draw[dotted,thick] (2.7225,0.6013)--(2.7225,0.6327);

\draw[thick,variable=\q,domain=2.732:2.891] plot({\q},({1/\q/(pow(\q,2)-1)+1/(\q-1)}); 
\draw[dotted, blue](2.732,0.634)--(2.732,-0.01);
\draw[dotted, blue](2.891,0.576)--(2.891,-0.01);
\draw(2.8115,-0.06)node[anchor=south] {$I^{12}$};

\draw[dotted,thick] (2.9122,0.5792)--(2.9122,0.6566);

\draw[thick,variable=\q,domain=2.917:2.971] plot({\q}, {(pow(\q,-2)-1)/(\q-1)/(pow(\q,-3)-1)-1/\q+1/(\q-1)}); 

\draw[dotted,thick] (2.9738,0.6385)--(2.9738,0.6637);

\draw[thick,variable=\q,domain=2.974:2.991] plot({\q}, {(pow(\q,-3)-1)/(\q-1)/(pow(\q,-4)-1)-1/\q+1/(\q-1)}); 


\draw[thick,variable=\q,domain=3:3.512] plot({\q},{1/(pow(\q,2)-\q)+1/(\q-1)}); 
\draw[dotted, blue](3,2/3)--(3,-0.01);
\draw[dotted, blue](3.512,0.511)--(3.512,-0.01);
\draw(3.256,-0.06)node[anchor=south] {$I^{2}$};

\draw[dotted,thick] (3.7321,0.5358)--(3.7321,0.732);

\draw[thick,variable=\q,domain=3.525:3.558] plot({\q}, {(\q+1)/pow(\q,3)+(2*pow(\q,2)+1)/(pow(\q,3)*(pow(\q,3)-1))+1/(\q-1)}); 

\draw[thick,variable=\q,domain=3.562:3.677] plot({\q},{1/pow(\q,2)+2/(\q*(pow(\q,2)-1))+1/(\q-1)}); 
\draw[dotted, blue](3.562,0.517)--(3.562,-0.01);
\draw[dotted, blue](3.677,0.491)--(3.677,-0.01);
\draw(3.6195,-0.06)node[anchor=south] {$I^{13}$};

\draw[dotted,thick] (3.6866,0.492)--(3.6866,0.5311);

\draw[thick,variable=\q,domain=3.689:3.72] plot({\q},{2/pow(\q,2)+(2*\q+1)/(pow(\q,2)*(pow(\q,3)-1))+1/(\q-1)}); 



\draw[thick,variable=\q,domain=3.7473:3.788] plot({\q},({1/\q/(pow(\q,3)-1)+2/(\q-1)}); 

\draw[thick,variable=\q,domain=3.791:3.939] plot({\q},({1/\q/(pow(\q,2)-1)+2/(\q-1)}); 
\draw[dotted, blue](3.791,0.736)--(3.791,-0.01);
\draw[dotted, blue](3.939,0.698)--(3.939,-0.01);
\draw(3.865,-0.06)node[anchor=south] {$I^{23}$};

\draw[dotted,thick] (3.9488,0.6989)--(3.9488,0.7468);

\draw[thick,variable=\q,domain=3.951:3.987] plot({\q}, {(pow(\q,-2)-1)/(\q-1)/(pow(\q,-3)-1)-1/\q+2/(\q-1)}); 


\end{tikzpicture}
\end{center}
\caption{The graph of the critical value function {$\tau(\beta)$ for} $\beta\in(1,4]$, with some of the basic intervals marked by blue dotted lines. (See Section \ref{sec:critical-value} for the definitions.) The interiors of the basic intervals make up the set $O$ in Theorem \ref{main:critical-devils-staircase}. Black dotted lines indicate jumps in the graph.}
\label{fig:graph-tau-more}
\end{figure}

\begin{remark}
In a previous version of this article (as well as in the published version), we remarked that Theorem \ref{main:critical-devils-staircase} (i) implies that $\dim_H K_\beta(t)$ is not jointly continuous in $\beta$ and $t$. This conclusion is not logically justified. We do not know whether the function $(\beta,t)\to \dim_H K_\beta(t)$ is jointly continuous.
\end{remark}


For the bifurcation set $\E_\beta$, we prove the following:

\begin{theorem} \label{thm:isolated-bifurcation-set}\mbox{} 
  \begin{enumerate}[{\rm(i)}]
	\item For any $\beta>1$,
	\[
	\E_\beta=\set{t\in[0,1): T_\beta^n(t)\ge t~\forall n\ge 0}.
	\]
	\item For any $\beta>1$, $\E_\beta$ is a Lebesgue null set of full Hausdorff dimension.
  \item For Lebesgue-almost every $\beta\in(1,\f)$, $\E_\beta$ contains both infinitely many isolated points and infinitely many accumulation points arbitrarily close to zero.
  \item There is an uncountable set $E_L$ of zero Hausdorff dimension such that ${\E_\beta}$ contains no isolated points if and only if $\beta\in E_L$.
  \end{enumerate}
\end{theorem}


\subsection{Related literature}

The sets $K_\beta(t)$ were first considered by Urba\'nski \cite{Urbanski-87} in the context of the doubling map ($\beta=2$). He proved that the function $\eta_2: t\mapsto \dim_H K_2(t)$ is a decreasing devil's staircase and the set-valued bifurcation set $\E_2$ is equal to the {\em dimension bifurcation set} $\BB_2$, where
\[
\BB_\beta:=\set{t\ge 0: \eta_\beta(t')<\eta_\beta(t)~\forall t'>t}, \qquad \beta>1.
\]
Urba\'{n}ski showed furthermore that the critical value $\tau(2)=1/2$ and that
\begin{equation*}
\dim_H\big(\E_2\cap[t,1]\big)=\dim_H K_2(t) \qquad \forall\,t\in(0,1).
\end{equation*}
Kalle {\em et al.}~\cite{Kalle-Kong-Langeveld-Li-18} conjectured that this identity should generalize to all $\beta\in(1,2]$, namely,
\begin{equation} \label{eq:dimension-identity}
\dim_H\big(\E_\beta\cap[t,1]\big)=\dim_H K_\beta(t) \qquad \forall\,t\in(0,1).
\end{equation}
Baker and Kong \cite{Baker-Kong-2020} verified \eqref{eq:dimension-identity} for the special case when $\beta$ is the positive root of $x^{n+1}=x^n+x^{n-1}+\dots+x+1$, where $n\in\N$. They also proved for such $\beta$ the equality $\E_\beta=\BB_\beta$. In our preprint \cite{Allaart-Kong-2023}, we provide a proof of \eqref{eq:dimension-identity} for all $\beta\in(1,2]$, and show that, in constrast to the findings in \cite{Baker-Kong-2020,Urbanski-87}, the difference $\E_\beta\backslash\BB_\beta$ has positive Hausdorff dimension for Lebesgue-almost every $\beta>1$, but can also have any finite cardinality or be countably infinite, depending on $\beta$. Using the extended Farey words introduced in this article, these results may be extended to all $\beta>1$; this will be done in a separate paper.

The doubling map with an arbitrary hole $(a,b)\subseteq [0,1)$ was considered by Glendinning and Sidorov \cite{Glendinning-Sidorov-2015}, and their work was partially extended to general $\beta$-transformations by Clark \cite{Clark-2016}. For $\beta\in(1,2]$, the symbolic survivor set corresponding to $K_\beta(t)$, which we introduce in Section \ref{sec:critical-values-proof}, can be put roughly in correspondence with the survivor set $K_2(a,b)$ for the doubling map with a suitable choice of hole $(a,b)$; hence there is a dictionary between our results from \cite{Allaart-Kong-2021} and those of Glendinning and Sidorov. Details of this relationship may be found in \cite{Allaart-Kong-2023}. However, this correspondence breaks down for $\beta>2$: The symbolic dynamics of the ``$k$-transformation" $T_k(x):=kx\pmod 1$ with an arbitrary hole for $k\in\N_{\geq 3}$, investigated by Agarwal \cite{Agarwal-2020}, are fundamentally different from those of the survivor set $K_\beta(t)$ for $k-1<\beta\leq k$ (even though the same alphabet is concerned), and there appears to be no direct link between the two problems. On the other hand, as pointed out by a referee, there is a direct connection between $K_\beta(t)$ for $\beta\in (1,k]$ with $k\geq 3$ and the map $T_k(x)$ with $k-1$ holes that are translated by multiples of $1/k$; we include a more precise remark at the end of the paper.

\bigskip

The mathematical study of open dynamical systems, or dynamical systems with holes, was first proposed by Pianigiani and Yorke \cite{Pianigiani-Yorke-1979}. Since then, numerous papers and books have been written on the subject. For instance, Urba\'nski \cite{Urbanski_1986} considered $C^2$-expanding, orientation-preserving circle maps with a hole $(0,t)$, and Dettmann \cite{Dettmann-2013} looked at asymptotics of circle maps with small holes. Besides the survivor set, another important object of study is the {\em escape rate}, i.e., the rate at which orbits of points enter the hole. Bunimovich and Yurchenko \cite{Bunimovich-Yurchenko-2011} showed that in general, the escape rate depends not only on the size, but also on the position of the hole. That this is also true for the size of the survivor set follows from the work of Glendinning and Sidorov \cite{Glendinning-Sidorov-2015}. Other relevant works, some of which also address applications to billiards and mathematical physics, are the articles by Demers and Young \cite{Demers-Young-2006}, Demers, Wright and Young \cite{Demers-Wright-Young-2010} and Bruin, Demers and Todd \cite{Bruin-Demers-Todd-2018}. The list is far from complete, of course. For a good overview of the subject, we refer to Chapter 8 of the book by Collet, Mart\'inez and San Mart\'in \cite{Collet-Martinez-SanMartin-2013}.   

\bigskip

\subsection{Organization of the paper}
This article is organized as follows. Section \ref{sec:prelim} introduces the extended Farey words and the substitutions of Farey words that underlie the main results. In Section \ref{sec:critical-value} we state our results more explicitly. We first introduce the {\em basic intervals}, whose interiors collectively make up the set $O$ in Theorem \ref{main:critical-devils-staircase}. We decompose the complement of $O$ into an {\em exceptional set} $E$, a countable collection of {\em relative exceptional sets} $\{E^\cs: \cs\in\La\}$ and an {\em infinitely renormalizable set} $E_\f$. Theorem \ref{thm:exceptional-sets} states that all these sets are small in terms of dimension. Theorem \ref{thm:critical-values} then gives explicit descriptions for the critical value $\tau(\beta)$, depending on whether $\beta$ lies in a basic interval, in $E$, in some $E^\cs (\cs\in\La)$, or in $E_\f$. Theorem \ref{prop:discontinuities}, whose somewhat tedious proof is essentially the same as in \cite{Allaart-Kong-2021} and is therefore omitted, identifies the discontinuity points of $\beta\mapsto \tau(\beta)$.  

Before proving the main results, we summarize known facts about greedy and quasi-greedy expansions in Section \ref{sec:greedy-and-quasi-greedy}. Theorems \ref{thm:exceptional-sets} and \ref{thm:critical-values} are proved in Sections \ref{sec:exceptional-set-proofs} and \ref{sec:critical-values-proof}, respectively. {Section \ref{sec:bifurcation-set} is devoted to the proof of Theorem \ref{thm:isolated-bifurcation-set}. Finally, in Section \ref{sec:connection} we describe the connection between $K_\beta(t)$ with $\beta\in(1,k]$ and the map $T_k: x\mapsto kx\pmod 1$ with $k-1$ holes.

\section{Preliminaries} \label{sec:prelim}

\subsection{Greedy and quasi-greedy expansions}

We begin by introducing important notation. First, for each $\beta>1$, we use the alphabet $A_\beta:=\{0,1,\dots,\lceil\beta\rceil-1\}$, where $\lceil \beta\rceil$ denotes the smallest integer greater than or equal to $\beta$. Note that the alphabet depends on $\beta$: for instance, for $\beta\in(1,2]$ we have $A_\beta=\{0,1\}$ whereas for $\beta\in(2,3]$ we have $A_\beta=\{0,1,2\}$. For a number $x\in[0,1)$, $b(x,\beta)$ will denote the {\em greedy expansion} of $x$ in base $\beta$ over the alphabet $A_\beta$, defined as the lexicographically largest sequence $(c_i)\in A_\beta^\N$ such that
\begin{equation} \label{eq:beta-expansion}
\pi_\beta((c_i)):=((c_i))_\beta:=\sum_{i=1}^\f \frac{c_i}{\beta^i}=x.
\end{equation}
Likewise, for $x\in[0,1]$, $a(x,\beta)$ denotes the {\em quasi-greedy expansion} of $x$ in base $\beta$ over $A_\beta$, defined as the lexicographically largest sequence $(c_i)\in A_\beta^\N$ that does not end in $0^\f$ and satisfies \eqref{eq:beta-expansion}. We write $\al(\beta):=a(1,\beta)$. Note that $a(x,\beta)=b(x,\beta)$ for all but countably many $x$.

\subsection{Farey words and extended Farey words}

Throughout the paper, we use the following notation. First, for any finite word $\w=w_1\dots w_m$ we write $\w^+:=w_1\dots w_{m-1}(w_m+1)$, and if $w_m>0$ we write similarly $\w^-:=w_1\dots w_{m-1}(w_m-1)$. When $m=1$ and $\w=w$, these definitions should be read as $\w^+=w+1$ and $\w^-=w-1$. Furthermore, we denote by $\L(\w)$ the lexicographically largest cyclical permutation of $\w$, so for instance, $\L(10210)=21010$.

We next introduce the Farey words. First we recursively define a sequence of ordered sets $F_n, n=0,1,2,\ldots$. Let $F_0=(0,1)$; and for $n\ge 0$ the ordered set $F_{n+1}=(\v_1, \ldots, \v_{2^{n+1}+1})$ is obtained from $F_{n}=(\w_1,\ldots, \w_{2^n+1})$ by
 \[
 \left\{
 \begin{array}
   {lll}
   \v_{2i-1}=\w_i&\textrm{for}& 1\le i\le 2^{n}+1,\\
   \v_{2i}=\w_i\w_{i+1}&\textrm{for}& 1\le i\le 2^n.
 \end{array}\right.
 \]
In other words, $F_{n+1}$ is obtained from $F_n$ by inserting for each $1\le j\le 2^{n}$ the new word $\w_j\w_{j+1}$ between the two neighboring words $\w_j$ and $\w_{j+1}$. So,
 \begin{equation*} 
 \begin{split}
 &F_1=(0,01,1),\qquad F_2=(0,001,01,011,1),\\
  F_3&=(0,0001,001,00101,01,01011,011,0111,1), \quad \dots
 \end{split}
 \end{equation*}
Note that for each $n\ge 0$ the ordered set $F_n$ consists of $2^n+1$ words which are listed from the left to the right in lexicographically increasing order.
We call $\w\in\set{0,1}^*$ a \emph{Farey word} if $\w\in F_n$ for some $n\ge 0$. We denote by $\cF:=\bigcup_{n=1}^\f F_n\backslash\{0,1\}$ the set of all Farey words of length at least two. Finally, we set 
\[
\hat{\cF}:=\cF\cup\{1\}.
\] 
For analyzing the case $\beta\geq 2$, the extra Farey word 1 will play a critical role.

Next, we define a map $\theta$ on finite words and infinite sequences that simply increments each coordinate by 1. Thus,
\begin{equation} \label{eq:theta}
\theta(c_1,\dots,c_n):=(c_1+1,\dots,c_n+1), \qquad \theta(c_1,c_2,\dots):=(c_1+1,c_2+1,\dots).
\end{equation}
We define the {\em extended Farey set} $\cF_e$ by
\[
\cF_e:=\{\theta^k(\w): \w\in \hat{\cF}, k=0,1,2,\dots\}.
\]
Thus, $\cF_e$ contains all the Farey words except $0$, and in addition, it contains all words derived from such Farey words by incrementing all digits by the same amount. For example, applying $\theta$ repeatedly to the Farey word $001$ yields the words $112, 223, 334, \dots$ in $\cF_e$.

The following property is what makes the (extended) Farey words special (see, e.g., \cite[Proposition 2.5]{Carminati-Isola-Tiozzo-2018}):

\begin{lemma} \label{lem:Farey-property}
Let $\s=s_1\ldots s_m\in\cF_e$. Then
\begin{enumerate}[{\rm (i)}]
\item $\L(\s)=s_ms_{m-1}\ldots s_1$.
\item ${\s^-}$ is a palindrome; that is $s_1\ldots s_{m-1}(s_m-1)=(s_m-1)s_{m-1}s_{m-2}\ldots s_1$.
\end{enumerate}
\end{lemma}

\subsection{The substitution operator}

We now recall the substitution operator $\bullet$ from \cite{Allaart-Kong-2021}. Say a finite word $\w$ is {\em Lyndon} if it is aperiodic and $\w$ is the lexicographically smallest among all its cyclical permutations. In particular, any word consisting of a single digit is Lyndon. We denote by $\cL_e$ the set of all Lyndon words except $0$, and by $\cL$ the set of all Lyndon words in $\{0,1\}^*$ of length at least 2. Then 
\[
\cF\subset\cL\subset\cL_e,\quad \cF_e\subseteq\cL_e,
\]
and words in $\cL_e$ never end in the digit $0$.

\begin{definition} \label{def:substitution}
For a Lyndon word $\s\in \cL_e$, we define the substitution map $\Phi_\s: \{0,1\}^\N\to \{\N\cup\{0\}\}^\N$ by
\begin{gather} \label{eq:block-map}
\begin{split}
\Phi_\s(0^{k_1}1^{l_1}0^{k_2}1^{l_2}\dots)=\s^-\L(\s)^{k_1-1}\L(\s)^+\s^{l_1-1}\s^-\L(\s)^{k_2-1}\L(\s)^+\s^{l_2-1}\dots,\\
\Phi_\s(1^{k_1}0^{l_1}1^{k_2}0^{l_2}\dots)=\L(\s)^+\s^{k_1-1}\s^-\L(\s)^{l_1-1}\L(\s)^+\s^{k_2-1}\s^-\L(\s)^{l_2-1}\dots,
\end{split}
\end{gather}
where $1\leq k_i, l_i\leq \f$ for all $i$.
We allow one of the exponents $k_i$ or $l_i$ to take the value $+\f$, in which case we ignore the remainder of the sequence. 
We define $\Phi_\s(\r)$ for a finite word $\r\in\set{0,1}^*$ in the same way. 
\end{definition}

The order in which the four blocks $\s,\s^-,\L(\s)$ and $\L(\s)^+$ can appear in $\Phi_\s(\r)$ is illustrated in Figure \ref{fig:directed-graph}.
   \begin{figure}[h!]
  \centering
  \begin{tikzpicture}[->,>=stealth',shorten >=1pt,auto,node distance=3cm, semithick,scale=3]

  \tikzstyle{every state}=[minimum size=0pt,fill=none,draw=black,text=black]

  \node[state] (A)                    { $\s$};
  \node[state]         (B) [ right of=A] {$\s^-$ };
  \node[state]         (C) [ above of=A] {$\L(\s)^+$};
  \node[state] (E)[left of=C]{Start-$1$};
  \node[state](D)[right of=C]{$\L(\s)$};
  \node[state](F)[right of=B]{Start-$0$};

  \path[->,every loop/.style={min distance=0mm, looseness=25}]
  (E) edge[->] node{$1$} (C)
  (C) edge[->,left] node{$1$} (A)
  (C) edge[bend right,->,right] node{$0$} (B)
(D) edge[loop right,->] node{$0$} (D)
(D) edge[->,above] node{$1$} (C)
(A) edge[loop left,->,looseness=55] node{$1$} (A)
(A)edge[->,below] node{$0$} (B)
(B) edge[bend right,->,right] node{$1$} (C)
(B) edge[->,right] node{$0$} (D)
(F) edge[->] node{$0$} (B)
;
\end{tikzpicture}
\caption{The directed graph illustrating the map $\Phi_\s$.}
\label{fig:directed-graph}
\end{figure}
For example,
\[
\Phi_\s(0110^\f)=\s^-\L(\s)^+\s\s^-\L(\s)^\f, \qquad \Phi_\s(11100)=\L(\s)^+\s^2\s^-\L(\s).
\]

Now for any two words $\s\in\cL_e$ and $\r\in\set{0,1}^*$ we define the substitution operation
\begin{equation} \label{eq:substitution}
\s\bullet\r:=\Phi_\s(\r).
\end{equation}

\begin{example} \label{ex:1}
Let $\s=01$ and $\r=011$. Then $\s^-=00$, $\L(\s)=10$, and $\L(\s)^+=11$. So,
\[
\s\bullet\r=\Phi_\s(\r)=\Phi_\s(011)={\s^-}\L(\s)^+\s=001101.
\]
On the other hand, if $\s=1$ and $\r=011$, we obtain $\s\bullet\r={\s^-}\L(\s)^+\s=021$. This shows that the operator $\Phi_1$ maps words from $\{0,1\}^*$ to words over the larger alphabet $\{0,1,2\}$. Similarly, we have, for example, $23\bullet 01=(23)^-\L(23)^+=2233$.
\end{example}

\begin{remark} \label{rem:substitution-remarks}
{\rm
(a) We have not defined expressions such as $01\bullet 021$, nor do we need to.

(b) The operation $\bullet$ commutes with the map $\theta$, in the sense that $\theta(\s\bullet\r)=\theta(\s)\bullet\r$. Thus, for $\s\in\cL_e$ and $\r\in\cL$, we can write $\s=\theta^k(\tilde{\s})$ for some $\tilde{\s}\in\cL\cup\{1\}$, and
$\s\bullet\r=\theta^k(\tilde{\s}\bullet\r)$.
}
\end{remark}

The next lemma collects some useful properties of the map $\Phi_\s$ and the substitution operator $\bullet$. For the proofs, see \cite[Section 3]{Allaart-Kong-2021}. Although the original proofs were given only for Lyndon words $\s$ in $\{0,1\}^*$, i.e., $\s\in\cL$, by using Remark \ref{rem:substitution-remarks} (b) it is easy to verify that all properties extend to $\s\in \cL_e$. In particular, the last property says that $\bullet$ is associative. In the lemma and throughout the paper, the symbols $\prec$, $\succ$, $\lle$ and $\lge$ are used to indicate the lexicographical order on words and sequences, defined in the usual way.

\begin{lemma} \label{lem:substitution-properties}
Let $\s\in\cL_e$.
\begin{enumerate}[{\rm(i)}]
\item The map $\Phi_\s$ is strictly increasing on $\set{0,1}^\N$.
\item For any  word $\d=d_1\ldots d_k\in \{0,1\}^*$ with $k\ge 2$, we have
\[
\left\{
\begin{array}{lll}
\Phi_\s(\d^-)=\Phi_\s(\d)^-&\textrm{if}& d_k=1,\\
\Phi_\s(\d^+)=\Phi_\s(\d)^+&\textrm{if}& d_k=0.
\end{array}\right.
\]
\item For any two sequences $\c,\d\in\set{0,1}^\N$, we have the equivalences
\[
\si^n(\c)\prec \d~\forall n\ge 0\quad\Longleftrightarrow\quad \si^n(\Phi_\s(\c))\prec \Phi_\s(\d)~\forall n\ge 0
\]
and
\[
\si^n(\c)\succ \d~\forall n\ge 0\quad\Longleftrightarrow\quad \si^n(\Phi_\s(\c))\succ \Phi_\s(\d)~\forall n\ge 0.
\]
\item For any $\r\in\cL$, we have $\s\bullet\r\in\cL_e$ and $\L(\s\bullet\r)=\s\bullet\L(\r)$.
\item For any $\r,\t\in\cL$, we have $(\s\bullet\r)\bullet\t=\s\bullet(\r\bullet\t)$.
\end{enumerate}
\end{lemma}

In \cite{Allaart-Kong-2021} we defined the sets
\[
\La(m):=\{\s_1\bullet\dots\bullet\s_m: \s_i\in \cF\ \mbox{for $i=1,\dots,m$}\}, \qquad m\in\N,
\]
and
\begin{equation} \label{eq:Lambda}
\La:=\bigcup_{m=1}^\f \La(m).
\end{equation}
We now extend this collection $\Lambda$ by setting
\[
\Lambda_e(m):=\{\s_1\bullet\s_2\bullet\dots\bullet \s_m: \s_1\in \cF_e, \s_2,\dots,s_m\in \cF\}, \qquad m\in\N,
\]
and
\begin{equation} \label{eq:Lambda_e}
\Lambda_e:=\bigcup_{m=1}^\f \Lambda_e(m).
\end{equation}
Observe that, in view of Remark \ref{rem:substitution-remarks} (b), $\Lambda_e$ can be written alternatively as
\[
\Lambda_e=\bigcup_{k=0}^\f \theta^k\big(\Lambda\cup \{1\} \cup \Phi_1(\Lambda)\big).
\]
Thus, the only truly new words (up to a repeated application of the map $\theta$) in the collection $\Lambda_e$ are $1$ and $1\bullet\cs$, $\cs\in\Lambda$.
However, throughout the paper we find it convenient to represent a generic element of $\Lambda_e$ as $\s_1\bullet\s_2\bullet\dots\bullet \s_m$, where $\s_1\in \cF_e$ and $\s_2,\dots,s_m\in \cF$.

\begin{remark} \label{rem:LMR-substitutions}
{\rm
Word $\cs\in\Lambda$ may be constructed alternatively and more elegantly using the three substitutions
\begin{equation} \label{eq:LMR}
L: \begin{cases}0 &\mapsto 0\\ 1 &\mapsto 10\end{cases}, \qquad
M: \begin{cases}0 &\mapsto 01\\ 1 &\mapsto 10\end{cases}, \qquad
R: \begin{cases}0 &\mapsto 01\\ 1 &\mapsto 1\end{cases}
\end{equation}
and the cyclical shift $\sigma_c$ defined on finite words by $\si_c(w_1w_2\dots w_n):=w_2\dots w_nw_1$. First, any Farey word $\s\in\cF$ can be expressed as $\s=\si_c\circ \varphi(1)$ for some $\varphi\in\{L,R\}^*M$, where $\{L,R\}^*$ is the free monoid over $\{L,R\}$ with composition of maps as the operation. We write this correspondence as $\s\leftrightarrow \varphi$. For example, $0010101 \leftrightarrow LRRM$. Similarly, if $\cs=\s_1\bullet\s_2\bullet\dots\bullet\s_k\in\La(k)$, we have $\cs \leftrightarrow \varphi_1\varphi_2\dots\varphi_k$, where each $\varphi_i\in \{L,R\}^*M$. For instance, $001\leftrightarrow  LM$, $011\leftrightarrow RM$, and hence $000\,101\,001=001\bullet 011 \leftrightarrow  LMRM$. So in general for each $\cs\in\La$ we have $\cs\leftrightarrow \varphi$ for some map $\varphi\in\{L,M,R\}^*M$. Furthermore, if $\cs\leftrightarrow \varphi$, it also holds that $\L(\cs)=\si_c\circ\varphi(0)$.
The substitutions $L, R$ and $M$ were used very effectively in \cite{Komornik-Steiner-Zou-2024}, and before that, with different notation, in \cite{Labarca-Moreira-2006}.

The reason why we did not define the words $\cs\in\La_e$ in terms of $L, R$ and $M$ is that the map $\Phi_1$ does not interact well with these substitutions. Because $\Phi_1$ maps words in $\{0,1\}^*$ to words in $\{0,1,2\}^*$, there is no simple substitution, defined just by its action on single letters, that can take its place. On the other hand, it is clear from the above definitions that $\Phi_1$ is a natural extension of the maps $\Phi_\cs$ for $\cs\in\Lambda$, and most proofs written in terms of the maps $\Phi_\cs$ continue to work for $\Phi_1$ as well. (As a more practical matter, we prefer to keep our notation consistent with our previous paper \cite{Allaart-Kong-2021}.)
}
\end{remark}

\section{The critical value $\tau(\beta)$} \label{sec:critical-value}

For each word $\cs\in\La_e$, we define three numbers $\beta_\ell^\cs$, $\beta_*^\cs$ and $\beta_r^\cs$ in $(1,\f)$ by
\[
(\L(\cs)^\f)_{\beta_\ell^\cs}=(\L(\cs)^+\cs^-\L(\cs)^\f)_{\beta_*^\cs}=(\L(\cs)^+\cs^\f)_{\beta_r^{\cs}}=1,
\]
and we define two intervals 
\[
I^\cs:=[\beta_\ell^\cs,\beta_*^\cs], \qquad J^\cs:=[\beta_\ell^\cs,\beta_r^\cs].
\]
Note that $I^\cs\subseteq J^\cs$. We call $I^\cs$ a {\em basic interval}, and for $\s\in\cF_e$, we call $J^\s$ a {\em Farey interval}. If $\cs=\s_1\bullet\dots\bullet \s_n$ with $\s_i\in\cF$ for {\em all} $i$ (including $i=1$), then $J^\cs$ lies inside the interval $(1,2)$ and, by associating $\cs$ with a substitution $\varphi\in\{L,M,R\}^*M$ as in Remark \ref{rem:LMR-substitutions}, we can write 
\[
\al(\beta_\ell^\cs)=\si(\varphi(0^\f)), \qquad \al(\beta_*^\cs)=\si(\varphi(010^\f)) \qquad \mbox{and} \qquad \al(\beta_r^\cs)=\si(\varphi(01^\f)).
\]
Note that in our notation we have $\al(\beta_\ell^\cs)=\L(\cs)^\f$, $\al(\beta_*^\cs)=\L(\cs)^+\cs^-\L(\cs)^\f=\Phi_\cs(10^\f)$ and $\al(\beta_r^\cs)=\L(\cs)^+\cs^\f=\Phi_\cs(1^\f)$.

On the other hand, if $\s_1\in\cF_e\setminus\cF$, then $J^\cs\subseteq [2,\f)$. In particular, for $\cs=k\in\N$ we have 
\[
\beta_\ell^\cs=\beta_\ell^{(k)}=k+1, \qquad \beta_r^\cs=\beta_r^{(k)}=\frac{k+2+\sqrt{k^2+4k}}{2},
\]
where we write $\beta_\ell^{(k)}, \beta_r^{(k)}$ instead of $\beta_\ell^k,\beta_r^k$ to avoid confusion with exponentiation.

\begin{proposition} \label{prop:disjoint}\mbox{}
\begin{enumerate} [{\rm(i)}]
\item The Farey intervals $J^\s$, $\s\in\cF_e$ are pairwise disjoint.
\item For each $\cs\in\La_e$, the intervals $J^{\cs\bullet\r}, \r\in\cF$ are pairwise disjoint subintervals of $J^\cs$.
\item The basic intervals $I^\cs, \cs\in\La_e$ are pairwise disjoint.
\end{enumerate}
\end{proposition}

We next define several special sets of bases, all of which are small in the sense of dimension. Set
\begin{gather*}
E:=(1,\f)\backslash \bigcup_{\s\in \cF_e} J^\s,\\
E^\cs:=(J^\cs\backslash I^\cs)\backslash \bigcup_{\r\in\cF} J^{\cs\bullet\r}, \qquad \cs\in\La_e.
\end{gather*}
and
\[
E_\f:=\bigcap_{m=1}^\f \bigcup_{\cs\in\La_e(m)} J^\cs.
\]
Note that $E_\f$ is the set of bases $\beta$ that belong to infinitely many of the intervals $J^\cs$, $\cs\in\La_e$. We call such bases {\em infinitely renormalizable}, because there is an infinite sequence $(\s_k)$ of (extended) Farey words such that $\beta\in J^{\s_1\bullet\dots\bullet\s_k}$ for each $k$, where $\s_1\in\cF_e$ and $\s_k\in\cF$ for all $k\geq 2$. (We call $(\s_k)$ the {\em coding} of $\beta$.)

Let $\dim_H$, $\dim_P$ and $\dim_B$ denote the Hausdorff, packing and box-counting dimensions, respectively.
The following result extends Propositions 5.6 and 5.8 of \cite{Allaart-Kong-2021}.

\begin{theorem} \label{thm:exceptional-sets}
The sets $E$, $E^\cs\, (\cs\in\La_e)$ and $E_\f$ are all zero-dimensional. Precisely,
\begin{enumerate} [{\rm(i)}]
\item $\dim_P E=0$;
\item $\dim_B E^\cs=0$ for every $\cs\in\La_e$;
\item $\dim_H E_\f=0$.
\end{enumerate}
\end{theorem}

In particular, the sets $E, E^\cs\, (\cs\in\La_e)$ and $E_\f$ all have Hausdorff dimension $0$. We call $E$ the {\em exceptional set} and the sets $E^\cs$ {\em relative exceptional sets}. In view of Proposition \ref{prop:disjoint}, we have the following decomposition of $(1,\f)$ into disjoint subsets:
\begin{equation} \label{eq:decomposition}
(1,\f)=E\cup \bigcup_{\cs\in\La_e} E^\cs \cup E_\f \cup \bigcup_{\cs\in\La_e} I^\cs.
\end{equation}

\begin{remark}
The set $E$ essentially extends the set called $E$ in our previous paper \cite{Allaart-Kong-2021}. Note, however, that in \cite{Allaart-Kong-2021} the set $E$ contained the base $2$ but here it does not, since $2$ is the left endpoint of the Farey interval $J^1$.
\end{remark}

Next we define, for each $\cs\in\La_e$, the map
\begin{equation} \label{eq:Psi-S}
\Psi_\cs: (1,2]\to  J^\cs\setminus I^\cs=(\beta_*^\cs, \beta_r^\cs];\quad \beta\mapsto \al^{-1}\circ\Phi_\cs\circ\al(\beta).
\end{equation}
In other words, $\beta=\Psi_\cs(\hat{\beta})$ if and only if $\al(\beta)=\Phi_\cs(\al(\hat{\beta}))$. {We will show in Section \ref{sec:exceptional-set-proofs} that
\begin{equation} \label{eq:E-transformation-1}
E^\cs=\Psi_\cs(E\cap(1,2))\cup\{\beta_r^\cs\}=\Psi_\cs\big((E\cap(1,2))\cup\{2\}\big), \qquad \cs\in\La_e.
\end{equation}
}
We can now completely specify the critical value $\tau(\beta)$ for all $\beta>1$.

\begin{theorem} \label{thm:critical-values} \mbox{}
\begin{enumerate} [{\rm(i)}]
\item If $I^\cs$ is a basic interval generated by $\cs\in\La_e$, then
   \begin{equation} \label{eq:tau-in-basic-interval}
   \tau(\beta)=(\cs^-\L(\cs)^\f)_\beta\qquad\textrm{for {every}}\ \beta\in I^\cs.
   \end{equation}
\item For all $\beta>1$ we have $\tau(\beta)\leq 1-(1/\beta)$. Furthermore, 
\[
\tau(\beta)=1-\frac{1}{\beta} \qquad\Longleftrightarrow \qquad \beta\in E_L:=E\cup\{\beta_\ell^\s: \s\in\cF_e\}.
\]
\item For any $\cs\in\La_e$ and $\beta\in E^\cs$ we have
\[
\tau(\beta)= \big(\Phi_\cs(0\hat{\al}_2\hat{\al}_3\ldots)\big)_\beta,
\]
where $1\hat{\al}_2\hat{\al}_3\ldots$ is the quasi-greedy expansion of $1$ in base $\hat\beta:=\Psi_\cs^{-1}(\beta)$. {(Note $\hat\beta$ is well defined in view of \eqref{eq:E-transformation-1}.)}
\item For any $\beta\in E_\f$ with coding $(\s_k)$ we have
\[
\tau(\beta) =\lim_{n\to\f} (\s_1\bullet\cdots\bullet\s_n 0^\f)_\beta.
\]
\end{enumerate}
\end{theorem}

In Figure \ref{fig:graph-tau-more}, we illustrate Theorem \ref{thm:critical-values} (i) by highlighting a few specific basic intervals. 
Note that $01\in\cF$, and the words $1,2=\theta(1), 12=\theta(01)$, and $23=\theta^2(01)$ lie in $\cF_e$. Moreover, $02=1\bullet 01\in\La_e(2)$ and $13=2\bullet 01=\theta(02)\in\La_e(2)$.

\begin{remark}
Part (i) of the above theorem extends \cite[Theorem 2]{Allaart-Kong-2021}; part (ii) extends a result of \cite{Kalle-Kong-Langeveld-Li-18}; and parts (iii) and (iv) extend Propositions 6.2 and 6.3 of \cite{Allaart-Kong-2021}, respectively. For the latter result, we also give a significantly simpler proof.
\end{remark}

To illustrate Theorem \ref{thm:critical-values} (iv), let $\beta_m$ be the {\em Komornik-Loreti constant} for the alphabet $\{0,1,\dots,m\}$, where $m\in\N$. That is, $\beta_m$ is the smallest base in which the number $1$ has a unique expansion over the alphabet $\{0,1,\dots,m\}$. In \cite{Allaart-Kong-2021} we showed that $\beta_1\in E_\f$ and $\tau(\beta_1)=(2-\beta_1)/(\beta_1-1)$. We give a slightly different proof of this fact here and extend the result to the other Komornik-Loreti constants.


\begin{proposition} \label{prop:K-L}
For each $m$, $\beta_m\in E_\f$ with coding $(\s_1,01,01,\dots)$, where
\[
\s_1=\begin{cases}
(k-1)k & \mbox{if $m=2k-1$},\\
k & \mbox{if $m=2k$},
\end{cases}
\]
and
\[
\tau(\beta_m)=\frac{m}{\beta_m-1}-1=\frac{m+1-\beta_m}{\beta_m-1}.
\]
\end{proposition}

\begin{proof}
It suffices to prove this for the case $k=1$ (i.e. $m=1$ and $m=2$), as the remaining cases are mere shifts of these two ``base" cases; see \cite{Komornik_Loreti_2002}. Consider first the case $m=1$. It is well known that $\al(\beta_1)=(\la_i)_{i=1}^\f$, where $(\lambda_i)_{i=0}^\f$ is the Thue-Morse sequence:
\[
(\lambda_i)_{i=0}^\f= 0110\,1001\,1001\,0110\,1001\,0110\,0110\,1001\dots.
\]
It is also well known that $(\lambda_i)_{i=0}^\f=M^\f(0):=\lim_{k\to\f}M^k(0)0^\f$, where $M$ is the substitution from \eqref{eq:LMR}. Thus, $\al(\beta_1)=\si(M^\f(0))$, which by Remark \ref{rem:LMR-substitutions} is equal to $\lim_{k\to\f}\L(\s_1\bullet\dots\bullet \s_k 0^\f)$ with $\s_i=01$ for all $i$. This implies that $\beta_1\in E_\f$ with coding $(01,01,\dots)$. Furthermore, since $M(1)=10=11-01=11-M(0)$ (with subtraction coordinatewise), we have
\[
\lim_{k\to\f}\s_1\bullet\dots\bullet\s_k 0^\f=\lim_{k\to\f}\si_c(M^k(1))0^\f=\si(M^\f(1))=1^\f-\si(M^\f(0)),
\]
giving, by Theorem \ref{thm:critical-values} (iv),
\[
\tau(\beta_1)=(1^\f)_{\beta_1}-\big(\si(M^\f(0))\big)_{\beta_1}=\frac{1}{1-\beta_1}-1.
\]

Moving on to the case $m=2$, it follows from \cite[Lemma 5.3]{Komornik_Loreti_2002} that 
\begin{align*}
\al(\beta_2)&=\Phi_1(\lambda_1\lambda_2\dots)=\Phi_1(\si(M^\f(0)))\\
&=\Phi_1(1101\,0011\,0010\,1101\dots)\\
&=2102\,0121\,0120\,2102\dots,
\end{align*}
and this implies immediately that $\beta_2\in E_\f$ with coding $(1,01,01,\dots)$. Furthermore, it is clear from the definition of $\Phi_1$ (see also Figure \ref{fig:directed-graph}) that for any sequence $(x_i)\in\{0,1\}^\N$,
\[
\Phi_1\big(1^\f-(x_i)\big)=2^\f-\Phi_1((x_i))
\]
(with subtraction again coordinatewise). Thus,
\begin{align*}
\lim_{k\to\f}\Phi_1(\underbrace{01\bullet\dots\bullet 01}_{k}\,0^\f)&=\Phi_1(\si(M^\f(1)))=\Phi_1\big(1^\f-\si(M^\f(0))\big)\\
&=2^\f-\Phi_1(\si(M^\f(0))),
\end{align*}
and so, by Theorem \ref{thm:critical-values} (iv),
\[
\tau(\beta_2)=(2^\f)_{\beta_2}-\big(\Phi_1(\si(M^\f(0)))\big)_{\beta_2}=\frac{2}{1-\beta}-1,
\]
as desired.
\end{proof}

\begin{remark}
{\rm
The application of the map $\Phi_1$ specifically to the Thue-Morse sequence occurs also in \cite{Allouche-Frougny-2009} and even already in \cite{Allouche-1983}, albeit with different notation.
}
\end{remark}

\begin{theorem} \label{prop:discontinuities}
The function $\tau$ is continuous at each point of $(1,\f)\backslash \{\beta_r^\cs: \cs\in\La_e\}$. On the other hand, for each $\cs\in\La_e$, we have
\begin{equation} \label{eq:upward-jump}
\lim_{\beta\searrow\beta_r^\cs}\tau(\beta)=(\cs^\f)_{\beta_r^\cs}>(\cs 0^\f)_{\beta_r^\cs}=\tau(\beta_r^\cs).
\end{equation}
\end{theorem}

(The proof of this theorem is the same as in \cite{Allaart-Kong-2021}.)

Theorem \ref{main:critical-devils-staircase} now follows from Theorem \ref{thm:exceptional-sets}, Theorem \ref{thm:critical-values} (i), Theorem \ref{prop:discontinuities} and the decomposition \eqref{eq:decomposition} by taking $O=\bigcup_{\cs\in\La_e}{\rm int}(I^\cs)$, where ${\rm int} (I^\cs)$ denotes the interior of $I^{\cs}$.

It is interesting to observe that for large $\beta$, the basic intervals $I^k=[\beta_\ell^{(k)},\beta_*^{(k)}]=[k+1,\beta_*^{(k)}]$ become dominant. This is made precise in the following proposition.

\begin{proposition} \label{prop:dominant-intervals}
For each $k\geq 2$ it holds that
\[
k+2-\frac{2}{k+2}<\beta_*^{(k)}<k+2.
\]
Hence, $\lim_{k\to\f}|I^k|=1$.
\end{proposition}

\begin{proof}
Let $\beta:=\beta_*^{(k)}$, and observe that $\al(\beta)=(k+1)(k-1)k^\f$. Hence
\[
\frac{k+1}{\beta}+\frac{k-1}{\beta^2}+\frac{k}{\beta^2(\beta-1)}=1,
\]
which leads to the cubic equation
\[
\beta^3-(k+2)\beta^2+2\beta-1=0.
\]
Denote the left side of this equation by $f_k(\beta)$. Then
\[
f_k'(\beta)=3\beta^2-2(k+2)\beta+2=\beta\{3\beta-2(k+2)\}+2>0 \qquad\mbox{for $\beta\geq k+1$},
\]
and $f_k(k+2)=2k+3>0$. A longer calculation shows that
\[
f_k\left(k+2-\frac{2}{k+2}\right)=-\frac{k(k^2+2k-4)}{(k+2)^3}<0 \qquad \mbox{for $k\geq 2$}.
\]
From these facts, the proposition follows.
\end{proof}

This result shows that for large $M\in\N$, the intersection of the exceptional set $E$ with $(M,M+1)$ lives entirely in a very small left neighborhood of $M+1$.

\section{Properties of greedy and quasi-greedy expansions} \label{sec:greedy-and-quasi-greedy}

We collect here some well-known characterizations and continuity properties of the greedy and quasi-greedy expansions. Since we are not working with a fixed alphabet, however, some extra care is needed with their statements.

\begin{lemma} \label{lem:quasi-greedy-expansion-of-1}
(See \cite{Baiocchi_Komornik_2007}.) Let $k\in\N$.
\begin{enumerate}[{\rm(i)}]
\item The restriction of the map $\beta\mapsto\al(\beta)$ to $(k,k+1]$ is an increasing bijection from $\beta\in (k,k+1]$ to the set of sequences $\a=a_1a_2\dots\in\set{0,1,\dots,k}^\N$ not ending with $0^\f$ such that $a_1=k$ and
\[
\si^n(\a)\lle \a\quad \forall n\ge 0.
\]
\item The map $\beta\mapsto\al(\beta)$ is left continuous everywhere on $(k,k+1]$ with respect to the order topology, and it is right continuous at $\beta_0\in(k,k+1)$ if and only if $\al(\beta_0)$ is not periodic. Furthermore, if $\al(\beta_0)=(a_1\ldots a_m)^\f$ with minimal period $m$, then $\al(\beta)\searrow a_1\ldots a_m^+ 0^\f$ as $\beta\searrow \beta_0$.
\end{enumerate}
\end{lemma}

The next result for greedy expansions was established in \cite{Parry_1960} and \cite[Lemma 2.5 and Proposition 2.6]{DeVries-Komornik-2011}.

\begin{lemma} \label{lem:greedy-expansion}
Let $\beta>1$. The map $t\mapsto b(t,\beta)$ is an increasing bijection from $[0,1)$ to the set
\[
\big\{\z\in A_\beta^\N: \si^n(\z)\prec \al(\beta)~\forall n\ge 0\big\}.
\]
Furthermore,
\begin{itemize}
  \item [{\rm(i)}] The map $t\mapsto b(t,\beta)$ is right-continuous everywhere in $[0,1)$ with respect to the order topology in $A_\beta^\N$;
  \item [{\rm(ii)}] If $b(t_0,\beta)$ does not end with $0^\f$, then the map $t\mapsto b(t,\beta)$ is continuous at $t_0$;
	\item [{\rm(iii)}] If $b(t_0,\beta)=b_1\ldots b_m 0^\f$ with $b_m>0$, then $b(t, \beta)\nearrow b_1\ldots b_m^-\al(\beta)$ as $t\nearrow t_0$.
\end{itemize}
\end{lemma}

Lemma \ref{lem:greedy-expansion} has the following analogue for quasi-greedy expansions.

\begin{lemma} \label{lem:quasi-greedy-expansion}
Let $\beta>1$. The map $t\mapsto a(t,\beta)$ is an increasing bijection from $(0,1]$ to the set
\[
\big\{\z\in A_\beta^\N: \si^n(\z)\lle \al(\beta)~\forall n\ge 0\big\}.
\]
Furthermore,
\begin{itemize}
  \item [{\rm(i)}] The map $t\mapsto a(t,\beta)$ is left-continuous everywhere in $(0,1]$ with respect to the order topology in $A_\beta^\N$;
  \item [{\rm(ii)}] If $a(t_0,\beta)$ does not end with $\al(\beta)$ (equivalently, if $b(t_0,\beta)$ does not end with $0^\f$), then the map $t\mapsto a(t,\beta)$ is continuous at $t_0$;
	\item [{\rm(iii)}] If $a(t_0,\beta)=a_1\dots a_m \al(\beta)$ (equivalently, if $b(t_0,\beta)=a_1\dots a_m^+0^\f$), then $a(t, \beta)\searrow a_1\dots a_m^+0^\f$ as $t\searrow t_0$.
\end{itemize}
\end{lemma}

\section{Proof of Theorem \ref{thm:exceptional-sets}} \label{sec:exceptional-set-proofs}

Recall the map $\theta$ from \eqref{eq:theta}. It induces a function $\phi:(1,\f)\to(2,\f)$ defined by
\[
\phi(\beta):=\al^{-1}\circ \theta \circ \al(\beta).
\]
By Lemma \ref{lem:quasi-greedy-expansion-of-1} the map $\phi$ is well defined and strictly increasing.
Furthermore, for each $k\in\N$, $\phi$ maps the interval $(k,k+1]$ into $(k+1,k+2]$. 
However, $\phi$ should not be confused with the map $\beta\mapsto \beta+1$. For instance,
\[
\phi\left(\frac{1+\sqrt{5}}{2}\right)=\al^{-1}\circ\theta((10)^\f)=\al^{-1}((21)^\f)=1+\sqrt{3}.
\]
In fact, for each $k\in\N$ we have
\[
 \lim_{\beta\searrow k} \phi(\beta)=\al^{-1}\circ\theta(k0^\f)=\al^{-1}((k+1)1^\f)=\frac{k+2+\sqrt{k^2+4}}{2},
\]
and hence, since $\phi$ is increasing, $\phi$ maps $(k,k+1]$ into $(\frac{k+2+\sqrt{k^2+4}}{2},k+2]$. 
In particular, $\phi$ does not map $(k,k+1]$ {\em onto} $(k+1,k+2]$.


We use the map $\phi$ to deduce some properties of the exceptional set $E$ from their known counterparts for $\beta\in(1,2]$, culminating in Proposition \ref{prop:E-is-zero-dimensional} below.

\begin{lemma} \label{non-overlapping-intervals}
For any two Lyndon words $\cs$ and $\cs'$, either $J^\cs$ and $J^{\cs'}$ are disjoint, or else one contains the other.
\end{lemma}

\begin{proof}
This was proved in \cite{Kalle-Kong-Langeveld-Li-18} for $J^\cs, J^{\cs'}\subseteq(1,2]$; the extension to $(1,\f)$ is trivial.
\end{proof}

\begin{lemma} \label{lem:two-digits}
Let $\beta\in E\cap(k,k+1]$ for $k\in\N_{\geq 2}$. Then $\al(\beta)\in\{k-1,k\}^\N$.
\end{lemma}

\begin{proof}
For $\beta\in(k,k+1]$ the alphabet is $A_\beta=\{0,1,\dots,k\}$, so if the conclusion were false there would be some integer $n\geq 0$ such that $\si^n(\al(\beta))\prec (k-1)^\f$.  Let $n_0\in\N$ be the smallest such $n$; then $\al_{n_0}=k$ and $\al_1\dots\al_{n_0}^-\in\{k-1,k\}^*$, where $\al(\beta)=\al_1\al_2\ldots$. Since $\beta\in E$ we have $\beta>\beta_r^{(k-1)}$ and so $\al(\beta)\succ \al(\beta_r^{(k-1)})=k(k-1)^\f$. This implies $n_0\geq 2$.

Let $\cs$ be the lexicographically smallest cyclic permutation of $\al_1\dots\al_{n_0}^-$. {Since $\al_{j+1}\dots\al_{n_0}^-\prec\al_{j+1}\dots\al_{n_0}\lle \al_1\dots\al_{n_0-j}$ for all $0\le j<n_0$, the word $\cs$ is not periodic, and hence it is Lyndon.} Furthermore, $|\cs|=n_0\geq 2$, and $\L(\cs)=\al_1\dots\al_{n_0}^-$. Clearly, $\al(\beta_\ell^\cs)=\L(\cs)^\f=(\al_1\dots\al_{n_0}^-)^\f\prec\al(\beta)$, so $\beta>\beta_\ell^\cs$. Furthermore, $\al(\beta_r^\cs)=\L(\cs)^+\cs^\f=\al_1\dots\al_{n_0}\cs^\f\succ\al(\beta)$ because $\al_{n_0+1}\al_{n_0+2}\dots\prec(k-1)^\f\lle\cs^\f$. Thus, $\beta<\beta_r^\cs$ and so $\beta\in J^\cs$. Although $\cs$ may not lie in $\cF_e$, we claim that there is a word $\cs'\in\cF_e$ such that $J^\cs\subseteq J^{\cs'}$.

Let $\tilde{\cs}:=\theta^{-(k-1)}(\cs)$. Then $\tilde{\cs}\in\{0,1\}^*$ because $\cs\in\{k-1,k\}^*$, and so $J^{\tilde{\cs}}\subseteq(1,2]$. As it was shown in \cite[Proposition 4.7]{Kalle-Kong-Langeveld-Li-18} that the intervals $J^{\r}: \r\in\cF$ are the maximal Lyndon intervals in $(1,2]$, Lemma \ref{non-overlapping-intervals} implies that $J^{\tilde\cs}\subseteq J^\r$ for some $\r\in\cF$. Put $\cs':=\theta^{k-1}(\r)$; then $\cs'\in\cF_e$ and $J^\cs\subseteq J^{\cs'}$. But this contradicts the assumption that $\beta\in E$.
\end{proof}

\begin{lemma} \label{prop:image-of-E}
For each $k\in\N_{\geq 2}$,
\[
E\cap(k,k+1]=\phi^{k-1}(E\cap(1,2]).
\]
\end{lemma}

\begin{proof}
If $\beta\in E\cap(1,2]$, then $\beta':=\phi^{k-1}(\beta)\in(\beta_r^{(k-1)},k+1]$. Suppose $\beta'\not\in E$; then $\beta'\in J^\s$ for some $\s\in \cF_e$. We cannot have $\s=k-1$ since $\beta'>\beta_r^{(k-1)}$. Hence $\s=\theta^{k-1}(\r)$ for some $\r\in\cF$. But then $\beta\in J^\r$, contradicting that $\beta\in E$. Therefore, $\beta'\in E\cap(k,k+1]$.

Conversely, let $\beta\in E\cap (k,k+1]$. Then $\al(\beta)\in\{k-1,k\}^\N$ by Lemma \ref{lem:two-digits}, and so $\al(\beta)=\theta^{k-1}(\al(\tilde{\beta}))$ for some $\tilde{\beta}\in(1,2]$. Equivalently, $\beta=\phi^{k-1}(\tilde{\beta})$. If $\tilde{\beta}\in J^\r$ for some $\r\in\cF$, then $\beta\in J^\s$, where $\s:=\phi^{k-1}(\r)$, because the map $\phi$ is increasing. This contradicts $\beta\in E$. Therefore, $\beta\in\phi^{k-1}(E\cap(1,2])$.
\end{proof}

We next prove that the restriction of $\phi^k$ to $E\cap[1+\ep,2]$ is Lipschitz for every $k\in\N$ and $\ep>0$. Then it follows immediately that the Hausdorff dimension of $E\cap(k+1,k+2]$ is zero for all $k\in\N$, since $\dim_H (E\cap(1,2])=0$ (see \cite[Proposition 5.6]{Allaart-Kong-2021}).

Equip the Baire space $(\N\cup\{0\})^\N$ with the metric
\[
\rho((c_i),(d_i)):=2^{-\inf\{i:c_i\neq d_i\}}.
\]

\begin{lemma} \label{lem:phi-Lipschitz}
For each $k\in\N$ and $\ep>0$, the restriction of $\phi^k$ to $E\cap[1+\ep,2]$ is H\"older continuous with exponent $\log\beta_r^{(k)}/\log 2$. In particular, it is Lipschitz.
\end{lemma}

\begin{proof}
Fix $\ep>0$. Then there is a positive integer $N$ such that, if $\beta\in[1+\ep,2]$, $\si(\al(\beta))$ begins with at most $N$ consecutive 0's. If furthermore $\beta\in E$, then $\al(\beta)$ is balanced (see \cite[Proposition 2.9]{Allaart-Kong-2023})  and hence it does not contain the word $0^{N+2}$.

Now let $\beta_1,\beta_2\in E$ such that $1+\ep\leq\beta_1<\beta_2\leq 2$. Then $\al(\beta_1)\prec\al(\beta_2)$, so there exists $n\in\N$ such that $\al_1(\beta_1)\dots\al_{n-1}(\beta_1)=\al_1(\beta_2)\dots\al_{n-1}(\beta_2)$ and $\al_n(\beta_1)<\al_n(\beta_2)$. By the previous paragraph, $\si^n(\al(\beta_2))\lge 0^{N+1}10^\f$. Hence, by Lemma \ref{lem:quasi-greedy-expansion-of-1} it follows that
\[
{\sum_{i=1}^n\frac{\al_i(\beta_2)}{\beta_1^i}\ge\sum_{i=1}^\f\frac{\al_i(\beta_1)}{\beta^i_1}=}1=\sum_{i=1}^\f \frac{\al_i(\beta_2)}{\beta_2^i}\geq \sum_{i=1}^n \frac{\al_i(\beta_2)}{\beta_2^i}+\frac{1}{\beta_2^{n+N+2}}.
\]
This gives
\[
\frac{1}{\beta_2^{n+N+2}}\leq \sum_{i=1}^n\left(\frac{\al_i(\beta_2)}{\beta_1^i}- \frac{\al_i(\beta_2)}{\beta_2^i}\right)\le\frac{1}{\beta_1-1}-\frac{1}{\beta_2-1}\le \frac{\beta_2-\beta_1}{\ep^2},
\]
and hence
\begin{equation} \label{eq:first-Lipschitz-inequality}
\rho\big(\al(\beta_1),\al(\beta_2)\big)=2^{-n}\leq \beta_2^{-n}\leq \beta_2^{N+2}\frac{\beta_2-\beta_1}{\ep^2}\leq \frac{2^{N+2}}{\ep^2}(\beta_2-\beta_1).
\end{equation}

Next, observe that $\inf \phi^k((1,2])=\beta_r^{(k)}$, because $\al(\beta_r^{(k)})=(k+1)k^\f=\theta^k(10^\f)$.
Put $\beta_1':=\phi^k(\beta_1)$ and $\beta_2':=\phi^k(\beta_2)$. Then $\al(\beta_1'),\al(\beta_2')\in\{k,k+1\}^\N$, and $\beta_r^{(k)}<\beta_1'<\beta_2'\leq k+2$. Let $n$ be as in the first part of the proof. Note that $\rho\big(\al(\beta_1'),\al(\beta_2')\big)=\rho(\al(\beta_1),\al(\beta_2))=2^{-n}$. Then, exactly as in the proof of \cite[Lemma 3.7]{Allaart-Baker-Kong-2019}, we obtain
\[
\beta_2'-\beta_1'\leq (\beta_2')^{2-n}\leq C_k\big(\beta_r^{(k)}\big)^{-n}=C_k\big[\rho\big(\al(\beta_1'),\al(\beta_2')\big)\big]^{\log\beta_r^{(k)}/\log 2},
\]
where $C_k=(k+2)^2$.
Combining this with \eqref{eq:first-Lipschitz-inequality} yields
\begin{align*}
\phi^k(\beta_2)-\phi^k(\beta_1)&=\beta_2'-\beta_1'\leq C_k\big[\rho(\al(\beta_1),\al(\beta_2))\big]^{\log\beta_r^{(k)}/\log 2}\\
&\leq C_k\left[\frac{2^{N+2}}{\ep^2}(\beta_2-\beta_1)\right]^{\log\beta_r^{(k)}/\log 2}=C_{k,\ep}(\beta_2-\beta_1)^{\log\beta_r^{(k)}/\log 2}
\end{align*}
for a constant $C_{k,\ep}$, and the proof is complete.
\end{proof}

\begin{proposition} \label{prop:E-is-zero-dimensional}
We have $\dim_P E=0$.
\end{proposition}

\begin{proof}
Fix $k\in\N$. We showed in \cite{Allaart-Kong-2021} that $\dim_B(E\cap(1+\ep,2])=0$ for every $\ep>0$. Thus, by Lemmas \ref{prop:image-of-E} and \ref{lem:phi-Lipschitz}, it follows that $\dim_B\big(E\cap(\beta_r^{(k)}+\ep,k+2]\big)=0$ for each $\ep>0$ and $k\in\N$, because Lipschitz images do not increase box dimension. Take a sequence $(\ep_n)$ decreasing to $0$. Then 
\[
E\cap(k+1,k+2]=E\cap(\beta_r^{(k)},k+2]=\bigcup_{n\in\N} \big(E\cap(\beta_r^{(k)}+\ep_n,k+2]\big),
\]
and hence $\dim_P(E\cap(k+1,k+2])=0$. Since packing dimension is countably stable, taking the union over $k\in\N$ yields $\dim_P E=0$.
\end{proof}

Next, for $\cs\in\La_e$ we recall the definition of the map $\Psi_\cs$ from \eqref{eq:Psi-S}.
Note that for each $\r\in\cF$, $\Psi_\cs(\beta_\ell^\r)=\beta_\ell^{\cs\bullet\r}$ and $\Psi_\cs(\beta_r^\r)=\beta_r^{\cs\bullet\r}$. Furthermore, $\Psi_\cs$ is strictly increasing. It follows as in the proof of \cite[Proposition 5.4]{Allaart-Kong-2021} that
\begin{equation} \label{eq:E-transformation}
E^\cs=\Psi_\cs(E\cap(1,2))\cup\{\beta_r^\cs\}.
\end{equation}
(Some care must be taken to ensure that, if $\beta\in E^\cs$, then $\beta$ lies in the range of $\Psi_\cs$. It turns out, however, that the gaps in the range of $\Psi_\cs$ are contained in the intervals $J^{\cs\bullet\r}, \r\in\cF$; see \cite{Allaart-Kong-2021} for the details. Note that, unlike in \cite{Allaart-Kong-2021}, here $2\not\in E$, so the above representation is slightly more awkward.)

\begin{lemma} \label{lem:Psi_S-Holder}
For each $\ep>0$, the restriction of $\Psi_\cs$ to $E\cap[1+\ep,2]$ is H\"older continuous with exponent $\log\beta_*^\cs/\log 2$.
\end{lemma}

\begin{proof}
The proof is nearly identical to that of Lemma \ref{lem:phi-Lipschitz}. Here we set $\beta_i':=\Psi_\cs(\beta_i)$ for $i=1,2$, and note that $\beta_*^\cs<\beta_1'<\beta_2'\leq\beta_r^\cs$. The rest of the proof proceeds in the same way.
\end{proof}

\begin{corollary} \label{cor:packing-dimension-of-E-S}
For each $\cs\in\La_e$, we have $\dim_P E^\cs=0$.
\end{corollary}

\begin{proof}
This follows from \eqref{eq:E-transformation} and Lemma \ref{lem:Psi_S-Holder} in essentially the same way as Proposition \ref{prop:E-is-zero-dimensional}.
\end{proof}

\begin{remark} \label{rem:box-dimension}
The slightly stronger statement that $\dim_B E^\cs=0$ follows in the same way as in \cite[Proposition 5.6 (ii)]{Allaart-Kong-2021}, and we omit the proof here. We must, however, rectify something at this point. In the proof of the above-mentioned proposition, we used the following fact (see \cite[Proposition 3.6]{Falconer_1997}): If $I,I_1,I_2,\dots$ are intervals such that $I_k\subseteq I$ for each $k$, $\sum_k|I_k|=|I|$, and $I_1,I_2,\dots$ are pairwise disjoint and arranged in order of decreasing length, then $\overline{\dim}_B (I\backslash \bigcup_k I_k)\leq 1/b$, where $b=-\limsup_{k\to\f}(\log|I_k|/\log k)$. Unfortunately, in \cite{Allaart-Kong-2021} we forgot to check that $\sum_{\r\in\cF}|J^{\cs\bullet\r}|=|J^\cs\backslash I^\cs|=\beta_r^\cs-\beta_*^\cs$. This is now a consequence of Corollary \ref{cor:packing-dimension-of-E-S}.
\end{remark}

\begin{proposition} \label{prop:E-infinity-small}
We have that $\dim_H E_\f=0$.
\end{proposition}

\begin{proof}
We showed in \cite[Proposition 5.8]{Allaart-Kong-2021} that $\dim_H(E_\f\cap(1,2])=0$. So if we can establish the equality
\begin{equation} \label{eq:E-infinity-self-similar}
E_\f=\bigcup_{\s\in\cF_e} \Psi_\s(E_\f\cap(1,2)),
\end{equation}
the proposition will follow from Lemma \ref{lem:Psi_S-Holder}, together with the countable stability of Hausdorff dimension, by an argument similar to that in the proof of Proposition \ref{prop:E-is-zero-dimensional}.

If $\beta\in E_\f$, then there exist a word $\s\in\cF_e$ and a sequence $(\r_n)$ in $\cF$ such that $\beta\in J^{\s\bullet\r_1\bullet\dots\bullet\r_n}$ for every $n$. Since $\beta$ is clearly not the left endpoint of such an interval, this implies that $\al(\beta)$ begins with $\L(\s\bullet\r_1\bullet\dots\bullet\r_n)^+=\Phi_\s\big(\L(\r_1\bullet\dots\bullet\r_n)^+\big)$ for each $n$. Since furthermore $|\r_1\bullet\dots\bullet\r_n|\to\f$ as $n\to\f$, it follows that
\[
\al(\beta)=\Phi_\s\left(\lim_{n\to\f}\L(\r_1\bullet\dots\bullet\r_n)^+0^\f\right),
\]
and so, $\beta=\Psi_\s(\hat{\beta})$ for some $\hat{\beta}\in(1,2)$. It is easy to see that $\hat{\beta}\in E_\f$ with coding $\r_1,\r_2,\dots$. Hence, $E_\f\subseteq \bigcup_{\s\in\cF_e} \Psi_\s(E_\f\cap(1,2))$.

Conversely, let $\beta\in \bigcup_{\s\in\cF_e} \Psi_\s(E_\f\cap(1,2))$; then $\beta=\Psi_\s(\hat{\beta})$ for some $\s\in\cF_e$ and $\hat{\beta}\in E_\f\cap(1,2)$. Let $(\r_n)\in\cF^\N$ be the coding of $\hat{\beta}$. Then $\beta\in E_\f$ with coding $\s,\r_1,\r_2, \ldots$. 
\end{proof}

Theorem \ref{thm:exceptional-sets} now follows from Proposition \ref{prop:E-is-zero-dimensional}, Remark \ref{rem:box-dimension} and Proposition \ref{prop:E-infinity-small}.

\section{Proof of Theorem \ref{thm:critical-values}} \label{sec:critical-values-proof}

Recall the definitions of the sets $\Lambda$ and $\Lambda_e$ from \eqref{eq:Lambda} and \eqref{eq:Lambda_e}, respectively.


\begin{proof}[Proof of Proposition \ref{prop:disjoint}]
Statement (i) was proved in \cite[Proposition 4.7]{Kalle-Kong-Langeveld-Li-18} for Farey intervals in $(1,2]$. The extension to $(1,\f)$ is easy.

For (ii), let $\cs\in\La_e$ and $\r_1,\r_2\in\cF$ with $\r_1\neq \r_2$. Then we have 
\[
\beta_\ell^{\cs\bullet\r_i}=\Psi_\cs(\beta_\ell^{\r_i}) \qquad\mbox{and}\qquad \beta_r^{\cs\bullet\r_i}=\Psi_\cs(\beta_r^{\r_i}) \qquad\mbox{for $i=1,2$}.
\]
Since $\Psi_\cs$ is strictly increasing and the intervals $J^{\r_1}=[\beta_\ell^{\r_1},\beta_r^{\r_1}]$ and $J^{r_2}=[\beta_\ell^{\r_2},\beta_r^{\r_2}]$ are disjoint by (i), it follows that the intervals $J^{\cs\bullet\r_1}$ and $J^{\cs\bullet\r_2}$ are disjoint as well.

Finally, we prove (iii). Given $\cs,\cs'\in\La_e$ with $\cs\neq\cs'$, the intervals $J^\cs$ and $J^{\cs'}$ are either disjoint or else one contains the other, by Lemma \ref{non-overlapping-intervals}. If they are disjoint, then of course $I^\cs$ and $I^{\cs'}$ are also disjoint. If, say, $J^{\cs'}\subseteq J^\cs$, then $\cs'=\cs\bullet \mathbf{R}$ for some $\mathbf{R}\in\La$, so $I^{\cs'}\subseteq J^{\cs'}\subseteq J^\cs\backslash I^\cs$, and again $I^\cs$ and $I^{\cs'}$ are disjoint.
\end{proof}

\begin{remark}
For $\cs\in\Lambda$, (ii) and (iii) were proved, respectively, in Proposition 5.1 and Remark 5.2 of \cite{Allaart-Kong-2021}.
\end{remark}

Recall that the survivor set $K_\beta(t)$ consists of all $x\in[0,1)$ whose orbit $\{T^n_\beta(x): n\ge 0\}$ avoids the hole $[0,t)$. To describe the dimension of $K_\beta(t)$ we introduce an analogous set in the symbolic space. The idea is as follows. By Lemma \ref{lem:greedy-expansion}, a sequence $\z\in A_\beta^\N$ is the greedy expansion of some element of $K_\beta(t)$ if and only if $\z$ belongs to the set
\[
\K_\beta(t):=\set{\z\in A_\beta^\N: b(t,\beta)\lle\si^n(\z)\prec \al(\beta)\ \forall n\ge 0}.
\]
This set is in general not a subshift, hence we enlarge it slightly to obtain the set
\begin{equation} \label{eq:K-tilde}
\widetilde{\K}_\beta(t):=\set{\z\in A_\beta^\N: b(t,\beta)\lle\si^n(\z)\lle \al(\beta)\ \forall n\ge 0}.
\end{equation}
Observe that $\widetilde{\K}_\beta(t)\backslash \K_\beta(t)$ is at most countable, and $\widetilde{\K}_\beta(t)$ is a subshift of $A_\beta^\N$ because it is closed and invariant under $\sigma$. (We could instead use the topological closure of $\K_\beta(t)$, which is also a subshift and is sometimes a proper subset of $\widetilde{\K}_\beta(t)$. But the explicit characterization of $\widetilde{\K}_\beta(t)$ makes this set easier to work with.)

For a subshift $X$ of the full shift $A_\beta^\N$, its \emph{topological entropy} $h_{top}(X)$ may be defined by
\[
h_{top}(X):=\lim_{n\to\f}\frac{\log N_n(X)}{n},
\]
where $N_n(X)$ denotes the number of words of length $n$ occurring in sequences of $X$. 
The following result can be deduced from Raith \cite{Raith-94} (see also \cite{Kalle-Kong-Langeveld-Li-18} for the case $\beta\in(1,2]$).

\begin{lemma} \label{lem:dim-survivorset}
Given $\beta>1$ and $t\in[0,1)$, the Hausdorff dimension of $K_\beta(t)$ is given by
\[
\dim_H K_\beta(t) =\frac{h_{top}(\widetilde{\K}_\beta(t))}{\log\beta}.
\]
\end{lemma}

Thus, it is sufficient to study the size of the symbolic survivor set $\widetilde{K}_\beta(t)$. The following lemma is crucial in this respect:

\begin{lemma} \label{lem:countable}
Let $\cs\in\La_e$, and let $M\in\N$ be such that $J^\cs\subseteq [M,M+1)$. Then the set
\begin{equation} \label{eq:Gamma}
\Ga(\cs):=\{\z\in \{0,1,\dots,M\}^\N: \cs^\f\lle \sigma^n(\z)\lle \L(\cs)^\f\ \forall n\geq 0\}
\end{equation}
is countable.
\end{lemma}

This was first proved in \cite[Proposition 4.1]{Allaart-Kong-2021} for the case when $\cs=\s_1\bullet\dots\bullet \s_n$ with each $\s_i\in \cF$ (corresponding to $M=1$). A much shorter proof can be found in \cite[Lemma 7.3]{Allaart-Kong-2023}. The proof goes through without change when $\s_1\in \cF_e$.

\begin{proof}[Proof of Theorem \ref{thm:critical-values} (i)]
The proof is mostly the same as that of \cite[Theorem 2]{Allaart-Kong-2021}, so we only give a sketch.
Fix a basic interval $I^\cs=[\beta_\ell^\cs, \beta_*^\cs]$.
Take $\beta\in I^\cs$, and let $t^*=(\cs^-\ca^\f)_\beta$, where $\ca=\L(\cs)=:a_1\ldots a_m$.  Then 
\begin{equation}  \label{eq:alpha-range}
  \ca^\f=\al(\beta_\ell^\cs)\lle \al(\beta)\lle \al(\beta_*^\cs)=\ca^+\cs^-\ca^\f.
\end{equation}

First we prove $\tau(\beta)\ge t^*$. We consider two cases: (a) $|\cs|\geq 2$, and (b) $|\cs|=1$.

(a) Suppose first that $|\cs|=m\geq 2$. Let $j<m$ be the integer such that $\cs=a_{j+1}\ldots a_m a_1\ldots a_j$, and for $N\in\N$ define $t_N:= \big((\cs^-\ca^N a_1\ldots a_j)^\f\big)_\beta$. It is not hard to verify that
\begin{equation*}
\si^n\big((\cs^-\ca^N a_1\ldots a_j)^\f\big)=\si^n\left(\cs^-\ca^{N+1}(a_1\ldots a_j^-\ca^{N+1})^\f\right)
\prec \ca^\f\lle\al(\beta)\qquad \forall n\ge 0.
\end{equation*}
So, by Lemma \ref{lem:greedy-expansion} we have $b(t_N, \beta)=(\cs^-\ca^N a_1\ldots a_j)^\f$. This implies that 
\[
\set{\cs^-\ca^{N+1}a_1\ldots a_j, \cs^- \ca^{N+2}a_1\ldots a_j}^\N\subseteq\K_\beta(t_N).
\]
Hence, $h_{top}(\widetilde{\K}_\beta(t_N))>0$, which by Lemma \ref{lem:dim-survivorset} implies that $\dim_H K_\beta(t_N)>0$. Thus $\tau(\beta)\ge t_N$ for all $N\ge 1$. Since $t_N\nearrow t^*$ as $N\to\f$, we conclude that $\tau(\beta)\ge t^*$.

(b) If $|\cs|=1$, then $\cs=\ca=k$ for some $k\in\N$. Without loss of generality we take $k=1$; the proof for bigger $k$ is the same up to a renaming of the digits. Here $\al(\beta)\lge 1^\f$, and $\cs^-\ca^\f=01^\f$. Put $t_N:=((01^N)^\f)_\beta$. Clearly, $b(t_N,\beta)=(01^N)^\f$. Furthermore, $\K_\beta(t_N)\supseteq\{01^{N+1},01^{N+2}\}^\N$, so $h_{top}(\widetilde{\K}_\beta(t_N))>0$. We conclude as above.

Next we prove $\tau(\beta)\le t^*$. From \eqref{eq:alpha-range} we obtain
\begin{align*}
\K_\beta(t^*)&\subseteq\set{\z\in A_\beta: \cs^-\ca^\f\lle \si^n(\z)\prec \ca^+\cs^-\ca^\f~\forall n\ge 0}\\
&=\set{\z\in A_\beta: \cs^-\ca^\f\lle \si^n(\z)\lle \ca^\f~\forall n\ge 0}=:\Ga.
\end{align*}
By Lemma \ref{lem:countable}, the set $\Ga(\cs)=\set{\z\in A_\beta: \cs^\f\lle\si^n(\z)\lle \ca^\f~\forall n\ge 0}$ is countable. Since any sequence in $\Ga\setminus\Ga(\cs)$ must end with $\cs^-\ca^\f$, $\Ga$ is also countable, and so is $\K_\beta(t^*)$. Thus, $\tau(\beta)\le t^*$.
\end{proof}

For the next lemma, recall that $a(t,\beta)$ and $b(t,\beta)$ denote the quasi-greedy and greedy expansions of $t$ in base $\beta$, respectively.

\begin{lemma} \label{lem:greedy-monotonicity}
Let $M\in\N$, and $M<\beta<\beta'\leq M+1$. Then:
\begin{enumerate}[{\rm(i)}]
\item $a(\tau(\beta),\beta))\lle {a(\tau(\beta'),\beta')}$;
\item If $b(\tau(\beta),\beta)$ is infinite, then also $b(\tau(\beta),\beta))\lle b(\tau(\beta'),\beta')$.
\end{enumerate}
\end{lemma}

\begin{proof}
(i) Observe first that, since $\beta<\beta'$, $a(\tau(\beta),\beta)$ is in fact the quasi-greedy $\beta'$-expansion of $(a(\tau(\beta),\beta))_{\beta'}$. Furthermore, by Lemma \ref{lem:quasi-greedy-expansion} the map $t\mapsto a(t,\beta')$ is strictly increasing. Since $a(t,\beta')=b(t,\beta')$ for all but countably many $t$, there is for each $\ep>0$ a number $t$ with
\[
(a(\tau(\beta),\beta))_{\beta'}-\ep<t<(a(\tau(\beta),\beta))_{\beta'}
\]
such that $b(t,\beta')=a(t,\beta')\prec a(\tau(\beta),\beta)$. Since the map $t\mapsto a(t,\beta)$ is in addition left-continuous by Lemma \ref{lem:quasi-greedy-expansion}, there exists a point $t_1<\tau(\beta)$ such that
\[
b(t,\beta')=a(t,\beta')\prec b(t_1,\beta)=a(t_1,\beta)\prec a(\tau(\beta),\beta).
\]
We then have (noting that $A_\beta=A_{\beta'}=\{0,1,\dots,M\}$)
\begin{align*}
\K_{\beta'}(t)&=\{\z\in A_{\beta'}^\N: b(t,\beta')\lle \si^n(\z)\prec\al(\beta')\ \forall n\geq 0\}\\
&\supseteq \{\z\in A_\beta^\N: b(t_1,\beta)\lle \si^n(\z)\prec\al(\beta)\ \forall n\geq 0\}\\
&=\K_\beta(t_1).
\end{align*}
Since $t_1<\tau(\beta)$, this shows, via Lemma \ref{lem:dim-survivorset}, that $\dim_H K_{\beta'}(t)>0$. Hence $\tau(\beta')\geq t$. Letting $\ep\searrow 0$ we have that $t\nearrow (a(\tau(\beta),\beta))_{\beta'}$. Therefore, $a(\tau(\beta'),\beta')\lge a(\tau(\beta),\beta)$.

(ii) Since $b(\tau(\beta),\beta)$ is infinite, the map $t\mapsto b(t,\beta)$ is continuous at $\tau(\beta)$. Thus, if $t<(b(\tau(\beta),\beta))_{\beta'}$, we can find a point $t_1<\tau(\beta)$ such that $b(t,\beta')\prec b(t_1,\beta)\prec b(\tau(\beta),\beta)$. As in the proof of (i) above, it follows that $\dim_H K_{\beta'}(t)>0$. Hence, $\tau(\beta')\geq (b(\tau(\beta),\beta)_{\beta'}$ since $b(\tau(\beta),\beta)$ is the greedy $\beta'$-expansion of $(b(\tau(\beta),\beta))_{\beta'}$. This implies $b(\tau(\beta'),\beta')\lge b(\tau(\beta),\beta)$.
\end{proof}

\begin{remark}
The conclusion of (ii) fails when $b(\tau(\beta),\beta)$ is finite: Take $\beta=\beta_\ell^\cs$ for any $\cs\in\La_e$, and $\beta'\in (\beta_\ell^\cs,\beta_*^\cs]$. Then $b(\tau(\beta),\beta)=\cs 0^\f\succ \cs^-\L(\cs)^\f=b(\tau(\beta'),\beta')$.
\end{remark}

\begin{lemma} \label{lem:tight-bounds}
For each $\cs\in\La_e$ and $\beta\in J^\cs$, we have
\[
\cs^-\L(\cs)^\f\lle b(\tau(\beta),\beta)\lle \cs 0^\f.
\]
\end{lemma}

\begin{proof}
If $\beta=\beta_\ell^\cs$, then $b(\tau(\beta),\beta)=\cs 0^\f$ by Theorem \ref{thm:critical-values} (i). Suppose $\beta\in (\beta_\ell^\cs,\beta_r^\cs]$. Then $\al(\beta)$ begins with $\L(\cs)^+$, so both $\cs^-\L(\cs)^\f$ and $\cs 0^\f$ are greedy $\beta$-expansions by Lemma \ref{lem:greedy-expansion}. Thus, the first inequality follows immediately from Lemma \ref{lem:greedy-monotonicity} (ii), by considering any point $\tilde{\beta}\in(\beta_\ell^\cs,\beta)$ and noting that $b(\tau(\tilde{\beta}),\tilde{\beta})=\cs^-\L(\cs)^\f$.
On the other hand, if $t\geq (\cs 0^\f)_\beta$, then
\begin{align*}
\K_\beta(t)&\subseteq \{\z\in A_\beta^\N: \cs 0^\f\lle \si^n(\z)\prec \L(\cs)^+ \cs^\f\ \forall n\geq 0\}\\
&=\{\z\in A_\beta^\N: \cs^\f\lle \si^n(\z)\lle \L(\cs)^\f\ \forall n\geq 0\},
\end{align*}
so $\K_\beta(t)$ is countable by Lemma \ref{lem:countable}. Hence, $\tau(\beta)\leq (\cs 0^\f)_\beta$, which implies $b(\tau(\beta),\beta)\lle \cs 0^\f$.
\end{proof}

\begin{proof}[Proof of Theorem \ref{thm:critical-values} (ii)]
We first show that $\tau(\beta)=1-\frac{1}{\beta}$ for $\beta\in E_L$. If $\beta=\beta_\ell^\s$ for some $\s\in\cF_e$, then $\tau(\beta)=(\s^-\L(\s)^\f)_\beta=(\L(\s)^\f)_\beta-\frac{1}{\beta}=1-\frac{1}{\beta}$, where the first equality follows from Theorem \ref{thm:critical-values} (i) and the second from Lemma \ref{lem:Farey-property}.

Suppose now that $\beta\in E$. Since $\dim_H E=0$, the intervals $J^\r: \r\in \cF_e$ are dense in $(1,\f)$, so there is a sequence $(\r_n)$ in $\cF_e$ such that $\beta_*^{\r_n}\nearrow\beta$. Note that
\[
\al\big(\beta_*^{\r_n}\big)=\L(\r_n)^+\r_n^-\L(\r_n)^\f,
\]
and
\[
b\left(\tau(\beta_*^{\r_n}),\beta_*^{\r_n}\right)=\r_n^-\L(\r_n)^\f.
\]
We may assume that $\beta_*^{\r_n}$ and $\beta$ lie in the same interval $(M,M+1]$, where $M\in\N$.
Therefore, by Lemma \ref{lem:greedy-monotonicity} (ii),
\begin{equation} \label{eq:b-lower-estimate-1}
b(\tau(\beta),\beta)\lge \r_n^-\L(\r_n)^\f.
\end{equation}
Next, observe that
\begin{equation*}
1\al_2\al_3\dots:=\al(\beta)=\lim_{n\to\f}\al(\beta_*^{\r_n})=\lim_{n\to\f}\L(\r_n)^+\r_n^-\L(\r_n)^\f,
\end{equation*}
where the second equality follows from the left continuity of the map $\gamma\mapsto \al(\gamma)$; see Lemma \ref{lem:quasi-greedy-expansion-of-1} (ii).
Since each $\r_n$ is a Farey word and $|\r_n|\to\f$ as $n\to\f$, this implies by Lemma \ref{lem:Farey-property} that $\lim_{n\to\f}\r_n=0\al_2\al_3\dots$. Hence, from \eqref{eq:b-lower-estimate-1} we obtain
\begin{equation} \label{eq:b-lower-inequality-1}
b(\tau(\beta),\beta)\lge 0\al_2\al_3\dots.
\end{equation}
Now we observe that $0\al_2\al_3\dots$ is a greedy $\beta$-expansion, since $\si^n(0\al_2\al_3\dots)\prec\al_1\al_2\dots=\al(\beta)$ for all $n\geq 0$, with strict inequality because $\al(\beta)$ is not periodic. Hence, \eqref{eq:b-lower-inequality-1} implies
\[
\tau(\beta)\geq (0\al_2\al_3\dots)_\beta=1-\frac{1}{\beta}.
\]
For the reverse inequality, note that there is a sequence $(\r_n')$ of Farey words such that $\beta_*^{\r_n'}\searrow\beta$. By the same reasoning as above but this time using Lemma \ref{lem:greedy-monotonicity} (i), we obtain
\[
a(\tau(\beta),\beta)\lle 0\al_2\al_3\dots=a\left(1-\frac{1}{\beta},\beta\right).
\]
Hence, $\tau(\beta)\leq 1-\frac{1}{\beta}$.

Next, suppose $\beta\not\in E_L$. Then $\beta\in J^\s\backslash\{\beta_\ell^\s\}$ for some $\s\in \cF_e$, so $\al(\beta)\succ \L(\s)^+0^\f$. By Lemma \ref{lem:tight-bounds}, $b(\tau(\beta),\beta)\lle \s 0^\f$. As a result,
\[
\tau(\beta)\leq (\s 0^\f)_\beta=(\L(\s)^+ 0^\f)_\beta-\frac{1}{\beta}<1-\frac{1}{\beta},
\]
where the equality follows from Lemma \ref{lem:Farey-property}.
\end{proof}

\begin{proof}[Proof of Theorem \ref{thm:critical-values} (iii)]
Let $\beta\in E^\cs$. Since $\dim_B E^\cs=0$, the intervals $J^{\cs\bullet \r}: \r\in\cF$ are dense in $J^\cs\backslash I^\cs$, hence there is a sequence $(\r_n)$ of Farey words such that $\beta_*^{\cs\bullet\r_n}\nearrow\beta$. Note that
\[
\al\big(\beta_*^{\cs\bullet\r_n}\big)=\L(\cs\bullet\r_n)^+(\cs\bullet\r_n)^-\L(\cs\bullet\r_n)^\f=\Phi_\cs\left(\L(\r_n)^+\r_n^-\L(\r_n)^\f\right),
\]
and by Theorem \ref{thm:critical-values} (i),
\[
b\left(\tau(\beta_*^{\cs\bullet\r_n}),\beta_*^{\cs\bullet\r_n}\right)=(\cs\bullet\r_n)^-\L(\cs\bullet\r_n)^\f=\Phi_\cs(\r_n^-\L(\r_n)^\f).
\]
Therefore, by Lemma \ref{lem:greedy-monotonicity} (ii),
\begin{equation} \label{eq:b-lower-estimate}
b(\tau(\beta),\beta)\lge \Phi_\cs(\r_n^-\L(\r_n)^\f).
\end{equation}
Next, the continuity of $\Phi_\cs$ and the left continuity of $\alpha$ imply
\begin{align*}
\Phi_\cs(1\al_2\al_3\dots)&=\Phi_\cs(\al(\hat{\beta}))=\al(\beta)=\lim_{n\to\f}\al(\beta_*^{\cs\bullet\r_n})\\
&=\lim_{n\to\f}\Phi_\cs\big(\L(\r_n)^+\r_n^-\L(\r_n)^\f\big)=\Phi_\cs\big(\lim_{n\to\f}\L(\r_n)^+\r_n^-\L(\r_n)^\f\big).
\end{align*}
Recalling that $\Phi_\cs$ is injective, it follows that $\lim_{n\to\f}\L(\r_n)^+\r_n^-\L(\r_n)^\f=1\al_2\al_3\dots$. Since each $\r_n$ is a Farey word and $|\r_n|\to\f$ as $n\to\f$, this implies by Lemma \ref{lem:Farey-property} that $\lim_{n\to\f}\r_n=0\al_2\al_3\dots$. Hence, from \eqref{eq:b-lower-estimate} and the continuity of $\Phi_\cs$ we obtain
\begin{equation} \label{eq:b-lower-inequality}
b(\tau(\beta),\beta)\lge \Phi_\cs(0\al_2\al_3\dots).
\end{equation}

From here, we consider two cases:

\medskip
{\em Case 1.} $\beta<\beta_r^\cs$. Then $\hat{\beta}<2$, and so $0\al_2\al_3\dots$ is a greedy $\hat{\beta}$-expansion by Lemma \ref{lem:greedy-expansion}, since $\si^n(0\al_2\al_3\dots)\prec\al_1\al_2\dots=\al(\hat{\beta})$ for all $n\geq 0$, with strict inequality because $\al(\hat{\beta})$ is not periodic. This implies $\Phi_\cs(0\al_2\al_3\dots)$ is also a greedy $\beta$-expansion. Hence, \eqref{eq:b-lower-inequality} implies
\[
\tau(\beta)\geq \left(\Phi_\cs(0\al_2\al_3\dots)\right)_\beta.
\]
For the reverse inequality, note that there is a sequence $(\r_n')$ of Farey words such that $\beta_*^{\cs\bullet\r_n'}\searrow\beta$. By the same reasoning as above but this time using Lemma \ref{lem:greedy-monotonicity} (i), we obtain
\[
a(\tau(\beta),\beta)\lle \Phi_\cs(0\al_2\al_3\dots).
\]
Letting $t^*:=\big(\Phi_\cs(0\al_2\al_3\dots)\big)_\beta$, we have $a(t^*,\beta)=b(t^*,\beta)=\Phi_\cs(0\al_2\al_3\dots)$. Hence $\tau(\beta)\leq t^*$.

\medskip
{\em Case 2.} $\beta=\beta_r^\cs$. Then $\hat{\beta}=2$, so $0\al_2\al_3\dots=01^\f$, and \eqref{eq:b-lower-inequality} gives
\[
b(\tau(\beta),\beta)\lge \Phi_\cs(01^\f)=\cs^-\L(\cs)^+\cs^\f.
\]
But since $\al(\beta)=\al(\beta_r^\cs)=\L(\cs)^+\cs^\f$, this implies $b(\tau(\beta),\beta)\lge \cs 0^\f$. The reverse inequality follows from Lemma \ref{lem:tight-bounds}. Thus, $b(\tau(\beta),\beta)=\cs 0^\f$ and so
\[
\tau(\beta)=(\cs 0^\f)_\beta=\big(\cs^-\L(\cs)^+\cs^\f\big)_\beta=\left(\Phi_\cs(0\al_2\al_3\dots)\right)_\beta.
\]
This completes the proof.
\end{proof}

\begin{proof}[Proof of Theorem \ref{thm:critical-values} (iv)]
Put $\cs_n:=\s_1\bullet\s_2\bullet\dots\bullet\s_n$. Then $\beta\in J^{\cs_n}$ for each $n$, so by Lemma \ref{lem:tight-bounds},
\[
\cs_n^-\L(\cs_n)^\f\lle b(\tau(\beta),\beta)\lle \cs_n 0^\f.
\]
Letting $n\to\f$ gives $b(\tau(\beta),\beta)=\lim_{n\to\f}\cs_n 0^\f$, from which the theorem follows.
\end{proof}

\section{The set-valued bifurcation set} \label{sec:bifurcation-set}

Recall from the Introduction the \emph{set-valued bifurcation set}
\begin{equation*} 
\E_\beta:=\set{t\in[0,1): K_\beta(t')\ne K_\beta(t)~\forall t'>t}.
\end{equation*}

In order to prove Theorem \ref{thm:isolated-bifurcation-set}, we need a few lemmas. 
The first one can be proved in the same way as \cite[Proposition 3.3]{Kalle-Kong-Langeveld-Li-18}.




\begin{lemma} \label{lem:isolated-1}
Let $\beta>1$. If $t$ is an isolated point of $\E_\beta$, then its greedy expansion $b(t,\beta)$ is periodic.
\end{lemma}

The next result can be deduced similarly to \cite[Proposition 3.10]{Kalle-Kong-Langeveld-Li-18}. Recall that for a Lyndon word $\s\in\mathcal L_e$, the Lyndon interval generated by $\s$ is $J^{\s}=[\beta_\ell^\s, \beta_r^\s]$, where
\[
\al(\beta_\ell^\s)=\L(\s)^\f,\quad\al(\beta_r^\s)=\L(\s)^+\s^\f.
\]

\begin{lemma} \label{lem:isolated-2}
  Let $J^\s=[\beta_\ell^\s, \beta_r^\s]$ be a Lyndon interval generated by $\s\in\mathcal L_e$.
  \begin{enumerate}[{\rm(i)}]
  \item $\si^n(\s^\f)\prec \al(\beta)$ for all $n\geq 0$ if and only if $\beta>\beta_\ell^\s$;
  \item If $\beta\in(\beta_\ell^\s, \beta_r^\s]$, then $(\s^\f)_\beta$ is isolated in $\E_\beta$;
  \item If $\beta>\beta_r^\s$, then $(\s^\f)_\beta$ is not isolated in $\E_\beta$.
  \end{enumerate}
\end{lemma}

Finally, we need the following property of $\al(\beta)$ from \cite{Schmeling-97}.

\begin{lemma} \label{lem:isolated-3}
For Lebesgue-almost every $\beta\in(1,\f)$, the quasi-greedy expansion $\al(\beta)$ contains arbitrarily long strings of consecutive zeros.
\end{lemma}

\begin{proof}[Proof of Theorem \ref{thm:isolated-bifurcation-set}]
Statement (i) is proved in \cite[Lemma 3.1]{Allaart-Kong-2023}.
Statement (ii) follows by an easy adaptation of the proof of \cite[Proposition 2.7]{Kalle-Kong-Langeveld-Li-18}. In fact, the argument in that proof shows that $\dim_H (\E_\beta\cap[0,\de])=1$ for any $\de>0$ and $\beta>1$. This implies that $\E_\beta$ contains infinitely many accumulation points arbitrarily close to zero for every $\beta>1$. Hence, to prove (iii) it suffices to show that $\E_\beta$ contains infinitely many isolated points in any neighborhood of $0$ for Lebesgue-almost all $\beta>1$.

Take $\beta>1$ such that $\al(\beta)$ contains arbitrarily long strings of consecutive zeros. By Lemma \ref{lem:isolated-3} it suffices to prove that $\E_\beta$ contains a sequence of isolated points decreasing to zero. Write
  \begin{equation}  \label{eq:iso-1}
    \al(\beta)=\b_1 0^{m_1}\b_2 0^{m_2}\cdots\b_k 0^{m_k}\b_{k+1}0^{m_{k+1}}\cdots,
  \end{equation}
  where each $\b_i$ contains no digit $0$, and each $m_i\in\N=\set{1,2,\ldots}$. Since $\al(\beta)$ contains arbitrarily long strings of consecutive zeros, we have $\sup_k m_k=+\f$.

  Set $i_0=1$, and let $i_1>i_0$ be the smallest index such that $m_{i_1}>m_1$. Then $m_{i_1}>m_j$ for all $j<i_1$. Set
  \[
	\a_1:=\b_10^{m_1}\ldots \b_{i_1}^-,
	\]
and let $\s_1:=\S(\a_1):=$ the lexicographically smallest cyclical permutation of $\a_1$.	
Since $\si^n(\al(\beta))\lle \al(\beta)$ for all $n\ge 0$, it follows that $\a_1$ is not periodic, and so $\s_1$ is Lyndon. Furthermore, by the definition of $i_1$ it follows that $\s_1$ begins with $0^{m_{1}}d_1$ for some digit $d_1>0$. We claim that
   \[
	\beta_\ell^{\s_1}<\beta<\beta_r^{\s_1},
	\]
where $[\beta_\ell^{\s_1}, \beta_r^{\s_1}]$ is the Lyndon interval generated by $\s_1$.
Note that $\al(\beta_\ell^{\s_1})=\L(\s_1)^\f=\a_1^\f$ and $\al(\beta_r^{\s_1})=\L(\s_1)^+\s_1^\f=\a_1^+\s_1^\f$. By (\ref{eq:iso-1}) and the fact that $m_{i_1}>m_1$ it follows that
  \begin{align*}
    \a_1^\f=(\b_1 0^{m_1}\ldots \b_{i_1}^-)^\f\prec \al(\beta)&=\b_1 0^{m_1}\ldots \b_{i_1}0^{m_{i_1}}\ldots\\
    &\prec \b_1 0^{m_1}\ldots \b_{i_1}0^{m_1}10^\f \lle\a_1^+\s_1^\f.
  \end{align*}
Thus, $\beta\in(\beta_\ell^{\s_1}, \beta_r^{\s_1})$ by Lemma \ref{lem:quasi-greedy-expansion-of-1}, establishing the claim. By Lemma \ref{lem:isolated-2} (ii) it follows that $(\s_1^\f)_\beta$ is isolated in $\E_\beta$.

Proceeding inductively, suppose $i_1,\dots,i_{k-1}$ have been chosen. Let $i_k>i_{k-1}$ be the smallest index such that $m_{i_k}>m_{i_{k-1}}$. By induction, $m_{i_k}>m_j$ for all $j<i_k$. Set $\a_k=\b_1 0^{m_1}\ldots \b_{i_k}^-$, and $\s_k:=\S(\a_k)$. Then $\s_k$ is Lyndon and begins with $0^{m_{i_{k-1}}}d_k$ for some digit $d_k>0$. By the same argument as above we can show that $\beta\in(\beta_\ell^{\s_k}, \beta_r^{\s_k})$, and thus $(\s_k^\f)_\beta$ is isolated in $\E_\beta$.

This way we construct a sequence of Lyndon words $\s_k$, $k\ge 1$, such that $\beta\in(\beta_\ell^{\s_k}, \beta_r^{\s_k})$ and $(\s_k^\f)_\beta$ is isolated in $\E_\beta$ for all $k\ge 1$. Since $\s_k$ begins with $0^{m_{i_{k-1}}}$ and $m_{i_{k-1}}\to\f$ as $k\to\f$, it follows that $(\s_k^\f)_\beta\searrow 0$ as $k\to\f$. This proves (iii).

Finally we prove (iv). We will show that ${\E_\beta}$ contains no isolated points if and only if $\beta\in E_L:=E\cup\set{\beta_\ell^\s: \s\in\mathcal F_e}$.
If $\beta\in(1,\f)\setminus E_L$, then $\beta\in(\beta_\ell^\s,\beta_r^\s]$ for some $\s\in\mathcal F_e$, so Lemma \ref{lem:isolated-2} (ii) implies that $(\s^\f)_\beta$ is an isolated point of $\E_\beta$. 

Conversely, take $\beta>1$ and suppose $\E_\beta$ contains an isolated point $t$. Then by Lemma \ref{lem:isolated-1} the greedy expansion $b(t,\beta)$ is periodic, say $b(t,\beta)=(b_1\ldots b_m)^\f$ with $m$ the minimal period. Since $t\in\E_\beta$, it follows from part (i) of the theorem and Lemma \ref{lem:greedy-expansion} that
   \begin{equation} \label{eq:iso-2}
   (b_1\ldots b_m)^\f\lle\si^n((b_1\ldots b_m)^\f)\prec\al(\beta)\quad\textrm{for all }n\ge 0.
   \end{equation} 
In particular, $\b:=b_1\ldots b_m$ is a Lyndon word, hence it generates a Lyndon interval $J^\b=[\beta_\ell^\b, \beta_r^\b]$. By assumption, $t=(\b^\f)_\beta$ is isolated in $\E_\beta$. Then by (\ref{eq:iso-2}) and Lemma \ref{lem:isolated-2} (i) and (iii), it follows that $\beta\in(\beta_\ell^\b, \beta_r^\b]$. As in the last paragraph of the proof of Lemma \ref{lem:two-digits}, we conclude that there exists an $\s\in\mathcal F_e$ such that $\beta\in(\beta_\ell^\s, \beta_r^\s]$. This means $\beta\notin E_L$, completing the proof.
\end{proof}

\section{Connection with the ``times $k$" map with multiple holes}\label{sec:connection}

As pointed out by the referee, the survivor set $K_\beta(t)$ is closely related to the map $T_k: [0,1)\to[0,1)$ given by $T_k(x):=kx \pmod 1$, with $k-1$ holes that are translates of each other by multiples of $1/k$. Precisely, let $0<a<1/k<b<a+(1/k)$, and define the open intervals
\[
H_1:=(a,b), \qquad H_2:=\left(a+\frac{1}{k},b+\frac{1}{k}\right), \quad \dots, \quad H_{k-1}=\left(a+\frac{k-2}{k},b+\frac{k-2}{k}\right).
\]
(See Figure \ref{fig:k-map-H}.)
Let $H:=\bigcup_{i=1}^{k-1}H_i$, and consider the set
\[
K(a,b;k-1):=\{x\in[0,1): T_k^n(x)\not\in H\ \forall n\geq 0\}.
\]

\begin{figure}[h!]
\begin{center}
\begin{tikzpicture}[
    scale=5,
    axis/.style={very thick, ->},
    important line/.style={thick},
    dashed line/.style={dashed, thin},
    pile/.style={thick, ->, >=stealth', shorten <=2pt, shorten
    >=2pt},
    every node/.style={color=black}
    ]
		\draw[thick,<->](0,1.1) node[anchor=south]{$y$} -- (0,0) node[anchor=north]{$0$} --(1.1,0) node[anchor=west]{$x$};
		\draw (1,0) node[anchor=north]{$1$};
		\draw (0,1) node[anchor=east]{$1$};

    \draw[important line]  (0,1)--(1,1);
    \draw[important line] (1, 1)--(1, 0);
    \draw[important line, blue] (0,0)--({1/3},1);
    \draw[important line, blue] ({1/3}, 0)--({2/3},1);
    \draw[important line, blue] ({2/3}, 0)--(1, 1);
    
     \draw[very thick, red] ({0.25}, 0) node{\scriptsize  $($} --(0.45, 0) node{\scriptsize  $)$};
     \draw[very thick, red] ({0.25+1/3}, 0) node{\scriptsize  $($}--({0.45+1/3}, 0) node{\scriptsize  $)$};
		\draw(0.25,-0.065) node{$a$};
		\draw(0.45,-0.065) node{$b$};
    \draw[dashed line] ({1/3},1)--({1/3},0);
           \draw[dashed line] ({2/3},1)--({2/3},0);
         \node[] at (0.35,-0.12){$H_1$};
            \node[] at ({0.35+1/3}, -0.12){$H_2$};

%

\end{tikzpicture}
\caption{The open dynamical system $(T_3, [0,1), H)$ with $H=H_1\cup H_2$. Note that $H_2=H_1+\frac13$.}
\label{fig:k-map-H}
\end{center}
\end{figure} 
 
We will associate with the pair $(a,b)$ a base $\beta\in(1,k]$ and a point $t\in(0,1)$ such that $K(a,b;k-1)$ has the same Hausdorff dimension as the set $\widetilde{\K}_\beta(t)$ defined in \eqref{eq:K-tilde}.

Write $a=(0.a_1a_2\dots)_k$ and $b=(0.b_1b_2\dots)_k$, and put $\mathbf{a}:=a_1a_2\dots$ and $\mathbf{b}:=b_1b_2\dots$.
By identifying points in $[0,1)$ with their base $k$ expansions, $K(a,b;k-1)$ is essentially the same as the symbolic set
\[
\Omega_{\mathbf{a},\mathbf{b}}:=\left\{\z\in\{0,1,\dots,k-1\}^\N: \si^n(\z) \in \bigcup_{j=0}^{k-1}S_j\ \forall n\geq 0\right\},
\]
where $S_0,S_1,\dots,S_{k-1}$ are the lexicographical intervals
\begin{align*}
S_0&:=[0^\f,(a_i)],\\
S_j&:=[j b_2b_3\dots,j a_2a_3\dots] \quad\mbox{for $j=1,\dots,k-2$},\\
S_{k-1}&:=[(k-1)b_2b_3\dots,(k-1)^\f].
\end{align*}
(Points having two different base $k$ expansions play no role, as they clearly do not belong to $K(a,b;k-1)$.)
Put $\mathbf{a}':=a_2a_3\dots$ and $\mathbf{b}':=b_2b_3\dots$. If $a_2\leq b_2$, then it is not too difficult to see that $\Omega_{a,b}$ is only countable, so we will assume that $a_2>b_2$. For general sequences $\mathbf{c},\mathbf{d}\in\{0,1,\dots,k-1\}^\N$ with $\mathbf{c}\prec\mathbf{d}$, we define the set
\[
\Sigma_{\mathbf{c},\mathbf{d}}:=\{\z\in\{0,1,\dots,k-1\}^\N: \mathbf{c}\lle \si^n(\z)\lle \mathbf{d}\ \forall n\geq 0\},
\]
and observe that this is a subshift of $\{0,1,\dots,k-1\}^\N$. Note that the symbolic survivor subshift $\widetilde{\K}_\beta(t)$ from Section \ref{sec:critical-values-proof} is simply $\Sigma_{\mathbf{c},\mathbf{d}}$, where $\mathbf{c}=b(t,\beta)$ and $\mathbf{d}=\al(\beta)$. It can be shown that
\[
\Sigma_{\mathbf{b}',\mathbf{a}'}\subset \Omega_{\a,\b}
\]
and vice versa, any sequence in $\Omega_{\a,\b}$ is of the form $\w\z$, where $\z\in\Sigma_{\mathbf{b}',\mathbf{a}'}$ and
$\w$ is of one of the following forms: (i) $0^m$ for $m\geq 0$; (ii) $(k-1)^m$ for $m\geq 0$; (iii) $0^m j$ for $m\geq 0$ and $
j\in\{1,\dots,k-2\}$; or (iv) $(k-1)^m j$ for $m\geq 0$ and $j\in\{1,\dots,k-2\}$. The proof is similar to that of Theorem 2.5 (vi) in \cite{Komornik-Steiner-Zou-2024}, and is omitted here.
It follows that $\Omega_{\a,\b}$ has the same Hausdorff dimension (and topological entropy) as $\Sigma_{\b',\a'}$. Now let $\b'':=\min\Sigma_{\b',\a'}$ and $\a'':=\max\Sigma_{\b',\a'}$. It is easy to check that $\Sigma_{\b',\a'}=\Sigma_{\b'',\a''}$ and $\a''$ and $\b''$ satisfy the inequalities
\[
\b''\lle\si^n(\a'')\lle \a'' \qquad\mbox{and} \qquad \b''\lle\si^n(\b'')\lle \a'' \qquad\mbox{for all $n\geq 0$}.
\]
It is also clear from these inequalities that $\a''$ does not end in $0^\f$, hence by Lemma \ref{lem:quasi-greedy-expansion-of-1} $\a''=\al(\beta)$ for some base $\beta\in(1,k)$. Furthermore, $\b''$ is the quasi-greedy expansion of some point $t\in[0,1)$ in base $\beta$ by the analogue of Lemma \ref{lem:greedy-expansion} for quasi-greedy expansions. If in fact $\si^n(\b'')\prec \a''$ for all $n\geq 0$, then $\b''$ is also the greedy expansion of $t$ and we have precisely $\widetilde{\K}_\beta(t)=\Sigma_{\b'',\a''}=\Sigma_{\b',\a'}$. Otherwise, there is a smallest $n\geq 0$ such that $\si^n(\b'')=\a''=\al(\beta)$, and we replace $\b''$ with $\b''':=b_1''\dots b_{n-1}''(b_n''+1)0^\f$, i.e. the greedy expansion of $t$ in base $\beta$. Then $\widetilde{\K}_\beta(t)=\Sigma_{\b''',\a''}$. Now the only sequences in $\Sigma_{\b'',\a''}\backslash\Sigma_{\b''',\a''}$ are those that end in $\a''$, i.e. a countable set. Therefore, $\widetilde{\K}_\beta(t)$ is equal to $\Sigma_{\b',\a'}$ minus a countable set. In either case, we conclude that 
\[
\dim_H\widetilde{\K}_\beta(t)=\dim_H \Sigma_{\b',\a'}=\dim_H \Omega_{\a,\b}=\dim_H K(a,b;k-1).
\]
Thus, the problem of finding, for fixed $\beta$, the smallest $t$ such that $\dim_H K_\beta(t)=0$ is in some sense equivalent to the problem of finding, for fixed $a$, the smallest $b$ such that $\dim_H K(a,b;k-1)=0$ (or, by symmetry, to the problem of finding, for fixed $b$, the largest $a$ such that $\dim_H K(a,b;k-1)=0$).


\section*{Acknowledgments}
The authors are grateful to the referee for valuable comments which led to an improvement in the presentation of the paper.
The first author was partially supported by Simons Foundation grant \#709869.
The second author was supported by Chongqing NSF: CQYC20220511052 and Scientific Research Innovation Capacity Support Project for Young Faculty No.~ZYGXQNISKYCXNLZCXM-P2P.

\end{document}